\documentclass{birkjour}
\usepackage{eurosym}
\usepackage{amsfonts}
\usepackage{amsmath}

\newtheorem{thm}{Theorem}[section]
\newtheorem{cor}[thm]{Corollary}
\newtheorem{lem}[thm]{Lemma}

\theoremstyle{definition}
\newtheorem{defn}[thm]{Definition}
\theoremstyle{remark}

\newtheorem*{ex}{Example}
\numberwithin{equation}{section}

\begin{document}

\title[3-Parameter Generalized Quaternions]{3-Parameter Generalized
Quaternions}
\author{Tuncay Deniz \c{S}ent\"{u}rk*}
\address{Department of Mathematics, Institute of Science and Technology\\
	Kastamonu University\\
	Kastamonu, 37150\\
	Turkey}
\email{tuncaydenizsenturk@gmail.com}

\thanks{*Corresponding author.}

\author{Zafer \"{U}nal}
\address{Department of Mathematics, Faculty of Arts and Science\\
	Kastamonu University\\
	Kastamonu, 37150\\
	Turkey}
\email{zunal@kastamonu.edu.tr}
\subjclass{14A20; 14A22; 15A66; 70G55; 70G65}
\keywords{3-parameter generalized quaternion, Lie group, matrix representation of quaternions, Lie algebra, De Moivre's formula, Euler's formula}
\date{October 25, 2019}

\begin{abstract}
In this article, we give the most general form of the quaternions algebra
depending on 3-parameters. We define 3-parameter generalized quaternions
(3PGQs) and study on various properties and applications. Firstly we present
the definiton, the multiplication table and other properties of 3PGQs such
as addition-substraction, multiplication and multiplication by scalar
operations, unit and inverse elements, conjugate and norm. We give matrix
representation and Hamilton operators for 3PGQs.We get
polar represenation, De Moivre's and Euler's formulas with the matrix
representations for 3PGQs. Besides, we give relations among the powers of
the matrices associated with 3PGQs. Finally, Lie group and Lie algebra
are studied and their matrix representations are shown. Also the Lie
multiplication and the killing bilinear form are given.
\end{abstract}

\maketitle

%
%
%
%
%
%
%
%
%






\section{Introduction}

Irish mathematician Sir William Rowan Hamilton started working on the
complex numbers in 1830. Hamilton wanted to generalize these numbers.
Firstly, he wanted to express these numbers as composition of two imaginary
numbers and one real number. So in the beginning he hoped to expand the
complex numbers into 3-dimensional space. Although he could do addition and
subtraction with these triples, he could not define norm with these triples.
For years he thought about this issue and made various researches. Finally,
on 16 October 1843, he defined real quaternions as:%
\begin{equation*}
\mathbb{H}=\left\{ a+be_{1}+ce_{2}+de_{3}\mid a,b,c,d\in \mathbb{R}%
,e_{1}^{2}=e_{2}^{2}=e_{3}^{2}=-1,e_{1}e_{2}e_{3}=-1\right\}
\end{equation*}%
\cite{ham0, ham3, ham4, ham1, ham5, ham2, hal}. Also in \cite{war}, all the
properties of quaternions, quaternion algebra and applications are explained
by Ward. Following the identification of the real quaternions, in 1849,
split-quaternion, also known as para-quaternion, co-quaternion,
pseudo-quaternion in the literature, was defined by Sir James Cockle in \cite{coc}:%
\begin{equation*}
\mathbb{P}=\left\{ a+be_{1}+ce_{2}+de_{3}\mid a,b,c,d\in \mathbb{R}%
,-e_{1}^{2}=e_{2}^{2}=e_{3}^{2}=1,e_{1}e_{2}e_{3}=1\right\} .
\end{equation*}%
Cockle has brought a new perspective to the quaternions. Hamilton
quaternions with complex coefficients are called biquaternions. The
biquaternions was described by Sir William Clifford in 1871 \cite{clif}. In 1924 and 1928, Leonard Eugene Dickson and Lois Wilfred Griffiths wrote two articles on generalized quaternions \cite{dic, gri}. The set of generalized quaternions with two parameter:
\begin{equation*}
\begin{array}{r}
\mathbb{H}_{\lambda ,\mu }=\left\{ a+be_{1}+ce_{2}+de_{3}\mid
a,b,c,d,\lambda ,\mu \in \mathbb{R},e_{1}^{2}=-\lambda ,e_{2}^{2}=-\mu
,\right. \\ 
\left. e_{3}^{2}=-\lambda \mu ,e_{1}e_{2}e_{3}=-\lambda \mu \right\} .%
\end{array}%
\end{equation*}%
These quaternions are known as generalized quaternions in the literature.
Throughout the article, we will refer as 2-parameters generalized
quaternions (2PGQs) for shortness of the impressions. In the set of the 2PGQs, if
as \linebreak $\lambda =\mu =1$ is taken, then we obtain Hamilton
quaternions. If as $\lambda =-\mu =1$ is taken, then we achieve set of the
split-quaternions.

It is possible to see the effects of Hamilton's discovery, which is about
two centuries ago, in many areas from physics to computer graphics. In the
current literature, quaternions are also associated with number sequences.
These studies can be found in \cite{sen, ahm, tok, bil, tas, swa, hor}.

In this article, we will go far beyond the generalization mentioned above
and we will give the most general form of the quaternions algebra depending
on 3-parameters.

\section{3-Parameter Generalized Quaternions}

In this section, we define the 3-parameter generalized quaternions and form
the algebra, inspired by the work of Hamilton, Cockle, Dickson and Griffiths.

\begin{defn}
The following set is called set of 3-parameter generalized quaternions
(3PGQs): 
\begin{equation*}
\begin{array}{r}
\mathbb{K}=\{a_{0}+a_{1}e_{1}+a_{2}e_{2}+a_{3}e_{3}\mid a_{0},a_{1},a_{2},a_{3},\lambda _{1},\lambda
_{2},\lambda _{3}\in \mathbb{R},\text{ }e_{1}^{2}=-\lambda _{1}\lambda _{2},
\\ 
e_{2}^{2}=-\lambda _{1}\lambda _{3},\text{ }e_{3}^{2}=-\lambda _{2}\lambda
_{3},\text{\ }e_{1}e_{2}e_{3}=-\lambda _{1}\lambda _{2}\lambda _{3}\}.%
\end{array}%
\end{equation*}
\end{defn}

Each element $p=a_{0}+a_{1}e_{1}+a_{2}e_{2}+a_{3}e_{3}$ of the set $\mathbb{K}$ is called a $%
3$-parameter generalized quaternion (3PGQ). Here the real numbers $a_{0},a_{1},a_{2},a_{3}$
are called components of $p$.\ The base vectors $1,e_{1},e_{2},e_{3}$ of the
3PGQs comply with the following multiplication table:%
\begin{equation}
\begin{array}{c|rrrr}
\cdot & 1 & e_{1} & e_{2} & e_{3} \\ \hline
1 & 1 & e_{1} & e_{2} & e_{3} \\ 
e_{1} & e_{1} & -\lambda _{1}\lambda _{2} & \lambda _{1}e_{3} & -\lambda
_{2}e_{2} \\ 
e_{2} & e_{2} & -\lambda _{1}e_{3} & -\lambda _{1}\lambda _{3} & \lambda
_{3}e_{1} \\ 
e_{3} & e_{3} & \lambda _{2}e_{2} & -\lambda _{3}e_{1} & -\lambda
_{2}\lambda _{3}%
\end{array}
\notag
\end{equation}%
According to this multiplication table, $\mathbb{K}=Sp\left\{ 1,e_{1},e_{2},e_{3}\right\} $. \newline
Special cases:\newline
i. \hspace*{0.21cm} If $\lambda _{1}=1,$ \ $\lambda _{2}=\lambda ,$ \ $\lambda _{3}=\mu $%
, then the algebra of 2PGQs is obtained.\newline
ii. \hspace*{0.2cm} If $\lambda _{1}=1,$ \ $\lambda _{2}=1,$ \ $\lambda _{3}=-1$, then
gives us the algebra of split quaternions.\newline
iii.\hspace*{0.15cm} If $\lambda _{1}=1,$ \ $\lambda _{2}=1,$ \ $\lambda _{3}=1$, then
the algebra of Hamilton quaternions is achieved.\newline
iv. \hspace*{0.1cm}If $\lambda _{1}=1,$ \ $\lambda _{2}=1,$ \ $\lambda _{3}=0$, then
the algebra of semi-quaternions is attained.\newline
v. \hspace*{0.2cm}$\lambda _{1}=1,$ \ $\lambda _{2}=-1,$ \ $\lambda _{3}=0$ then we get
the algebra of split semi-quaternions.\newline
vi. \ \ \ $\lambda _{1}=1,$ \ $\lambda _{2}=0,$ \ $\lambda _{3}=0$ then
algebra of $1/4$-quaternions is achieved. \newline
Of course, it is possible to work in more specific quaternion algebras
according to $\lambda _{i\in \left\{ 1,2,3\right\} }$.

Throughout the article, we will consider special cases for the $\lambda _{i\in \left\{ 1,2,3\right\} }$ values given above.

 Any 3PGQ $%
p=a_{0}+a_{1}e_{1}+a_{2}e_{2}+a_{3}e_{3}$ consists of two parts, the scalar and the vector
part:%
\begin{equation*}
p=S_{p}+V_{p}
\end{equation*}%
where 
\begin{equation*}
S_{p}=a_{0}\text{ }\ \text{ve }V_{p}=a_{1}e_{1}+a_{2}e_{2}+a_{3}e_{3}.
\end{equation*}

\begin{defn}
Let $p$ be a 3PGQ. If $S_{p}=0$, then $\ p$ is called 3-parameter
generalized pure-quaternion (3PGPQ) or 3-parameter generalized vector
(3PGV). Let us show the set of 3-parameter generalized vectors is as
follows: 
\begin{equation*}
\text{Im}\left( \mathbb{K}\right) =\left\{
a_{1}e_{1}+a_{2}e_{2}+a_{3}e_{3}\mid a_{1},a_{2},a_{3}\in \mathbb{R}\right\}
.
\end{equation*}
\end{defn}

Equality, addition, multiplication by scalar and multiplication
operations are defined on $\mathbb{K}$ as following:

Let $p=a_{0}+a_{1}e_{1}+a_{2}e_{2}+a_{3}e_{3}$ and $%
q=b_{0}+b_{1}e_{1}+b_{2}e_{2}+b_{3}e_{3}$ be 3PGQs and $\alpha $ be a real
number.

Equality:$\ p=q\Leftrightarrow a_{0}=b_{0},$ $a_{1}=b_{1},$ $a_{2}=b_{2},$ $%
a_{3}=b_{3}$.

Addition: $p+q=\left( S_{p}+S_{q}\right) +\left( V_{p}+
V_{q}\right) =\left( a_{0}+b_{0}\right) +\left( a_{1}+b_{1}\right)
e_{1}+\left( a_{2}+b_{2}\right) e_{2}+\left( a_{3}+b_{3}\right) e_{3}$.

Multiplication by scalar: The following operation is called multiplication
by scalar or external operation: 
\begin{equation*}
\begin{array}{rrr}
\odot :\mathbb{R\times K} & \rightarrow & \mathbb{K\hspace*{6.12cm}} \\ 
(c ,p) & \rightarrow & c \odot p=:c p=c a_{0}+c
a_{1}e_{1}+c a_{2}e_{2}+c a_{3}e_{3}%
\end{array}%
\end{equation*}%
$\bigskip $Multiplication: 
\begin{equation*}
\begin{array}{rrr}
\times :\mathbb{K\times K} & \rightarrow & \mathbb{K}\hspace*{1.37cm} \\ 
(p,q) & \rightarrow & p\times q=pq%
\end{array}%
\end{equation*}%
if $p$ and $q$ is multiplied\ according to the multiplication table, then we
have: 
\begin{align*}
pq& =\left( a_{0}b_{0}-\lambda _{1}\lambda _{2}a_{1}b_{1}-\lambda
_{1}\lambda _{3}a_{2}b_{2}-\lambda _{2}\lambda _{3}a_{3}b_{3}\right) \\
& +e_{1}\left( a_{0}b_{1}+b_{0}a_{1}+\lambda _{3}\left(
a_{2}b_{3}-a_{3}b_{2}\right) \right) \\
& +e_{2}\left( a_{0}b_{2}+b_{0}a_{2}+\lambda _{2}\left(
a_{3}b_{1}-a_{1}b_{3}\right) \right) \\
& +e_{3}\left( a_{0}b_{3}+a_{3}b_{0}+\lambda _{1}\left(
a_{1}b_{2}-a_{2}b_{1}\right) \right) .
\end{align*}%
We can formulate this result as follows:%
\begin{align*}
pq& =\left( S_{p}+V_{p}\right) \left( S_{q}+V_{q}\right) \\
& =S_{p}S_{q}+S_{p}V_{q}+S_{q}V_{p}+V_{p}V_{q} \\
& =S_{p}S_{q}-f\left( V_{p},V_{q}\right) +S_{p}V_{p}+S_{q}V_{q}+V_{p}\wedge
V_{q},
\end{align*}%
where 
\begin{equation*}
\begin{array}{rrr}
f:\text{Im}\left( \mathbb{K}\right) \mathbb{\times }\text{Im}\left( \mathbb{K%
}\right) & \rightarrow & \mathbb{R\hspace*{6.65cm}} \\ 
\left( V_{p},V_{q}\right) & \rightarrow & f\left( V_{p},V_{q}\right)
=\lambda _{1}\lambda _{2}a_{1}b_{1}+\lambda _{1}\lambda
_{3}a_{2}b_{2}+\lambda _{2}\lambda _{3}a_{3}b_{3}%
\end{array}%
\end{equation*}%
and 
\begin{equation*}
\begin{array}{ll}
\wedge :\text{Im}\left( \mathbb{K}\right) \mathbb{\times }\text{Im}\left( 
\mathbb{K}\right) & \rightarrow \text{Im}\left( \mathbb{K}\right) \\ 
\hspace*{1.85cm}\left( V_{p},V_{q}\right) & \rightarrow V_{p}\wedge
V_{q}=\left\vert 
\begin{array}{ccc}
\lambda _{3}e_{1} & \lambda _{2}e_{2} & \lambda _{1}e_{3} \\ 
a_{1} & a_{2} & a_{3} \\ 
b_{1} & b_{2} & b_{3}%
\end{array}%
\right\vert \\ 
& \hspace*{1.55cm}%
\begin{array}{l}
=\lambda _{3}\left( a_{2}b_{3}-a_{3}b_{2}\right) e_{1} \\ 
+\lambda _{2}\left( a_{3}b_{1}-a_{1}b_{3}\right) e_{2}+\lambda _{1}\left(
a_{1}b_{2}-a_{2}b_{1}\right) e_{3}.%
\end{array}%
\end{array}%
\end{equation*}%
If $p=V_{p}=a_{1}e_{1}+a_{2}e_{2}+a_{3}e_{3}$ and $%
q=V_{q}=b_{1}e_{1}+b_{2}e_{2}+b_{3}e_{3}$ then the multiplication of $p$ and 
$q$ is: 
\begin{equation*}
\begin{array}{rrr}
\times :\mathbb{K\times K} & \rightarrow & \mathbb{K}\hspace*{5.9cm} \\ 
(V_{p},V_{q}) & \rightarrow & V_{p}\times V_{q}=:V_{p}V_{q}=-f\left(
V_{p},V_{q}\right) +V_{p}\wedge V_{q}%
\end{array}%
\end{equation*}%
There are two special cases:\newline
i.\ \ If $V_{p}\perp V_{q}$, then $V_{p}V_{q}=V_{p}\wedge V_{q},$\newline
ii.\hspace*{0.05cm} If $V_{p}\parallel V_{q}$, then $V_{p}V_{q}=-f\left(
V_{p},V_{q}\right) .$

\begin{thm}
Let $p,q$ and $r$ be 3PGVs. The following equations are satisfied:\newline
i. \ \ $p\wedge \left( q\wedge r\right) =f\left( p,r\right) q-f\left(
p,q\right) r,$\newline
ii. \ $\left( p\wedge q\right) \wedge r=f\left( p,r\right) q-f\left(
q,r\right) p.$
\end{thm}

\begin{proof}

i. If $p=a_{1}e_{1}+a_{2}e_{2}+a_{3}e_{3},$ $%
q=b_{1}e_{1}+b_{2}e_{2}+b_{3}e_{3},$ \newline $%
r=c_{1}e_{1}+c_{2}e_{2}+c_{3}e_{3}$ then we get%
\begin{align}
p\wedge \left( q\wedge r\right)
& =e_{1}\left( a_{2}\left( b_{1}c_{2}-b_{2}c_{1}\right) \lambda _{1}\lambda
_{3}+a_{3}\left( b_{1}c_{3}-b_{3}c_{1}\right) \lambda _{2}\lambda _{3}\right)
\notag \\
& +e_{2}\left( a_{1}\left( b_{2}c_{1}-b_{1}c_{2}\right) \lambda _{1}\lambda
_{2}+a_{3}\left( b_{2}c_{3}-b_{3}c_{2}\right) \lambda _{2}\lambda _{3}\right)
\notag \\
& +e_{3}\left( a_{1}\left( b_{3}c_{1}-b_{1}c_{3}\right) \lambda _{1}\lambda
_{2}+a_{2}\left( b_{3}c_{2}-b_{2}c_{3}\right) \lambda _{1}\lambda _{3}\right)
\ \label{312}
\end{align}%
on the other hand we have%
\begin{align}
f\left( p,r\right) q-f\left( p,q\right) r
& =e_{1}\left( x_{2}\left( y_{1}z_{2}-y_{2}z_{1}\right) \lambda _{1}\lambda
_{3}+x_{3}\left( y_{1}z_{3}-y_{3}z_{1}\right) \lambda _{2}\lambda _{3}\right)
\notag \\
& +e_{2}\left( x_{1}\left( y_{2}z_{1}-y_{1}z_{2}\right) \lambda _{1}\lambda
_{2}+x_{3}\left( y_{2}z_{3}-y_{3}z_{2}\right) \lambda _{2}\lambda _{3}\right)
\notag \\
& +e_{3}\left( x_{1}\left( y_{3}z_{1}-y_{1}z_{3}\right) \lambda _{1}\lambda
_{2}+x_{2}\left( y_{3}z_{2}-y_{2}z_{3}\right) \lambda _{1}\lambda _{3}\right)
\ \label{313}
\end{align}%
according to Eq.(\ref{312}) and Eq.(\ref{313}), the result is obtained.
\newline
ii. Similar to i, the existence of proof is seen.
\end{proof}

\begin{cor}
Let $p$ and $q$ be two 3PGVs. Then 
\begin{equation*}
S(pq)=-f(p,q).
\end{equation*}
\end{cor}

\begin{proof}
If $p=a_{1}e_{1}+a_{2}e_{2}+a_{3}e_{3}$ and $%
q=b_{1}e_{1}+b_{2}e_{2}+b_{3}e_{3}$, then%
\begin{eqnarray*}
	S(pq) =-\lambda _{1}\lambda _{2}a_{1}b_{1}-\lambda _{1}\lambda
	_{3}a_{2}b_{2}-\lambda _{2}\lambda _{3}a_{3}b_{3}=-f\left( p,q\right).	
\end{eqnarray*}
\end{proof}

\begin{cor}
i. \ $\left( \mathbb{K},+\right) $ is an Abelian group.\newline
ii. \ The abelian group $\left( \mathbb{K},+\right) $ is a vector space on
the field ${\mathbb{R}}$ with the external operation $\odot $.\newline
iii. \ $\left\{ \mathbb{K},+,\times \right\} $ is a ring with unity. \newline
iv. \hspace*{0.05cm} $\left\{ \mathbb{K},+,\times \right\} $ is not a
commutative ring\textbf{.}\newline
v. \hspace*{0.16cm} $\left\{ \mathbb{K},+,\times \right\} \ $\ is not an
integral domain.\newline
vi. \hspace*{0.06cm} $\left\{ \mathbb{K},+,\times \right\} $ is not a field.%
\newline
vii. \hspace*{-0.05cm} $\left\{ \mathbb{K},+,\mathbb{R},+,\cdot ,\odot
\right\} $ is a vector space\textit{.}\newline
viii. $\left\{ \mathbb{K},+,\mathbb{R},+,\cdot ,\odot ,\times \right\} $ is
an algebra\textit{.} This algebra is called \textit{3-parameter generalized
quaternion algebra.}
\end{cor}

\begin{proof}
The reader can easily prove to the all item.
\end{proof}

\begin{defn}
The conjugate of a 3PGQ $p$ is defined as follows 
\begin{equation*}
\begin{array}{rrr}
C:\mathbb{K} & \rightarrow & \mathbb{K}\hspace*{2.85cm} \\ 
p & \rightarrow & C(p)=:\bar{p}=S_{p}-V_{p}%
\end{array}%
\end{equation*}
\end{defn}

If $p=a_{0}+a_{1}e_{1}+a_{2}e_{2}+a_{3}e_{3}$, then $\bar{p}%
=a_{0}-a_{1}e_{1}-a_{2}e_{2}-a_{3}e_{3}$.

\begin{thm}
i. For all $p,q$ in $\mathbb{K}$ and all $c_{1},c_{2}$ in $\mathbb{R}$, $%
\overline{c_{1}p+c_{2}q}=\overline{c_{1}p}+\overline{c_{2}q}$,\newline
ii. For all $p,q$ in $\mathbb{K}$, $\overline{pq}=\bar{q}\bar{p}$,\newline
iii. For all $p$ in $\mathbb{K}$, $\overline{\overline{p}}=p$.
\end{thm}

\begin{proof}
i. For $p=a_{0}+a_{1}e_{1}+a_{2}e_{2}+a_{3}e_{3}$ and $%
q=b_{0}+b_{1}e_{1}+b_{2}e_{2}+b_{3}e_{3}$,
\begin{eqnarray*}
	\overline{c_{1}p+c_{2}q}
	&=&\left( c_{1}a_{0}+c_{2}b_{0}\right) -\left( c_{1}a_{1}+c_{2}b_{1}\right)
	e_{1}-\left( c_{1}a_{2}+c_{2}b_{2}\right) e_{2} \\
	&&-\left( c_{1}a_{3}+c_{2}b_{3}\right) e_{3} \\
	&=&c_{1}\left( a_{0}-a_{1}e_{1}-a_{2}e_{2}-a_{3}e_{3}\right) +c_{2}\left(
	b_{0}-b_{1}e_{1}-b_{2}e_{2}-b_{3}e_{3}\right) \\
	&=&\overline{c_{1}p}+\overline{c_{2}q}
\end{eqnarray*}%
ii and iii can be shown in a similar way.
\end{proof}

\begin{thm}
For any two 3PGVs $p,q$, 
\begin{equation*}
p\wedge q=\frac{q\bar{p}-p\bar{q}}{2}
\end{equation*}
is provided.
\end{thm}

\begin{proof} If $p=a_{1}e_{1}+a_{2}e_{2}+a_{3}e_{3}$ and $%
q=b_{1}e_{1}+b_{2}e_{2}+b_{3}e_{3}$, then
\begin{eqnarray*}
	p\wedge q &=&\left\vert 
	\begin{array}{ccc}
		\lambda _{3}e_{1} & \lambda _{2}e_{2} & \lambda _{1}e_{3} \\ 
		a_{1} & a_{2} & a_{3} \\ 
		b_{1} & b_{2} & b_{3}%
	\end{array}%
	\right\vert \\
	&=&\lambda _{3}\left( a_{2}b_{3}-a_{3}b_{2}\right) e_{1} \\
	&&+\lambda _{2}\left( a_{3}b_{1}-a_{1}b_{3}\right) e_{2}+\lambda _{1}\left(
	a_{1}b_{2}-a_{2}b_{1}\right) e_{3} \\
	&=&\frac{1}{2}\left( \bar{q}p-p\bar{q}\right) .
\end{eqnarray*}
\end{proof}

\begin{defn}
\begin{equation*}
\begin{array}{rrr}
N:\mathbb{K} & \rightarrow & \mathbb{R}\hspace*{1.85cm} \\ 
p & \rightarrow & N_{p}=p\overline{p}=\overline{p}p%
\end{array}%
\end{equation*}%
The funciton $N$ is called norm operation on $\mathbb{K}$. The norm of any
3PGQ $p$ is calculated as follows: 
\begin{equation}
N_{p}=p\overline{p}=a_{0}^{2}+\lambda _{1}\lambda _{2}a_{1}^{2}+\lambda
_{1}\lambda _{3}a_{2}^{2}+\lambda _{2}\lambda
_{3}a_{3}^{2}=S_{p}S_{p}+f(V_{p},V_{p}).  \label{320}
\end{equation}
\end{defn}

Let $p$ be 3PGQ. If $N_{p}=1$, then $p$ is called 3-parameter generalized
unit quaternion (3PGUQ).

\begin{thm}
For all $p,q$ in $\mathbb{K}$ and all $c$ in $\mathbb{R}$,\newline
i. \ \ $N_{p}N_{q}=N_{pq}$,\newline
ii. \ $N_{cp}=c^{2}N_{p}$.
\end{thm}

\begin{proof}
The reader can easily prove i and ii using Eq.(\ref{320})
\end{proof}

\begin{defn}
The following function is called an inverse operation on $\mathbb{K}$: 
\begin{equation*}
\begin{array}{rrr}
I:\mathbb{K} & \rightarrow & \mathbb{R}\hspace*{4.05cm} \\ 
\ p & \rightarrow & I\left( p\right) =:p^{-1}=\dfrac{\overline{p}}{N_{p}},%
\text{ }N_{p}\neq 0.%
\end{array}%
\end{equation*}%
Let $p$ be a nonzero 3PGQ. If $p=a_{0}+a_{1}e_{1}+a_{2}e_{2}+a_{3}e_{3}$,
then inverse of $p\ $ is as follows: 
\begin{equation}
p^{-1}=\frac{\overline{p}}{N_{p}}=\frac{%
a_{0}-a_{1}e_{1}-a_{2}e_{2}-a_{3}e_{3}}{a_{0}^{2}+\lambda _{1}\lambda
_{2}a_{1}^{2}+\lambda _{1}\lambda _{3}a_{2}^{2}+\lambda _{2}\lambda
_{3}a_{3}^{2}}.  \label{324}
\end{equation}
\end{defn}

\begin{thm}
For any two nonzero 3PGQs $p$ and $q$, any nonzero real number $c$, we have
the followings:

i. \ \ $\left( pq\right) ^{-1}=q^{-1}p^{-1},$

ii. \ $\left( cp\right) ^{-1}=\dfrac{1}{c}p^{-1}.$
\end{thm}

\begin{proof}
The proof can easily be  proved by using Eq.(\ref{324})
\end{proof}

\begin{defn}
Let $p=S_{p}+V_{p}$ and $q=S_{q}+V_{q}$ be any two 3PGQs. The multiplication
defined as follows is called the scalar multiplication of two 3PGQs: 
\begin{equation}
\begin{array}{rrr}
\left\langle ,\right\rangle :\mathbb{K\times K} & \rightarrow & \mathbb{R}%
\hspace*{1cm}\hspace*{1cm}\hspace*{1.6cm} \\ 
\left( p,q\right) & \rightarrow & \left\langle p,q\right\rangle
=S_{p}S_{q}+f\left( V_{p},V_{q}\right)%
\end{array}
\label{325}
\end{equation}
\end{defn}

Also if $p=a_{0}+a_{1}e_{1}+a_{2}e_{2}+a_{3}e_{3}$ and $%
q=b_{0}+b_{1}e_{1}+b_{2}e_{2}+b_{3}e_{3}$ then 
\begin{align*}
\left\langle p,q\right\rangle =a_{0}b_{0}+\lambda _{1}\lambda
_{2}a_{1}b_{1}+\lambda _{1}\lambda _{3}a_{2}b_{2}+\lambda _{2}\lambda
_{3}a_{3}b_{3}=S(p\bar{q}).
\end{align*}

\begin{lem}
For all $p,q$ in $\mathbb{K}$, on the metric in $\mathbb{K}$, $S(p\bar{q})=S(%
\bar{q}p)$.
\end{lem}

\begin{proof} 
For $p=a_{0}+a_{1}e_{1}+a_{2}e_{2}+a_{3}e_{3}$ and $%
q=b_{0}+b_{1}e_{1}+b_{2}e_{2}+b_{3}e_{3}\in \mathbb{K}$, we obtain
\begin{eqnarray*}
	p\bar{q} 
	&=&\left( a_{0}b_{0}+\lambda _{1}\lambda _{2}a_{1}b_{1}+\lambda _{1}\lambda
	_{3}a_{2}b_{2}+\lambda _{2}\lambda _{3}a_{3}b_{3}\right) \\
	&&+e_{1}\left( -a_{0}b_{1}+b_{0}a_{1}+\lambda _{3}\left(
	-a_{2}b_{3}+a_{3}b_{2}\right) \right) \\
	&&+e_{2}\left( -a_{0}b_{2}+b_{0}a_{2}+\lambda _{2}\left(
	-a_{3}b_{1}+a_{1}b_{3}\right) \right) \\
	&&+e_{3}\left( -a_{0}b_{3}+a_{3}b_{0}+\lambda _{1}\left(
	-a_{1}b_{2}+a_{2}b_{1}\right) \right)
\end{eqnarray*}%
$S(p\bar{q})=\left\langle p,q\right\rangle $ is seen. Also we find
\begin{align}
\left\langle q,p\right\rangle =b_{0}a_{0}+\lambda _{1}\lambda
_{2}b_{1}a_{1}+\lambda _{1}\lambda _{3}b_{2}a_{2}+\lambda _{2}\lambda
_{3}b_{3}a_{3}=S\left( \bar{q}p\right).  \label{327}
\end{align}%
The existence of proof is apparent from Eq.(\ref{325}) and Eq.(\ref{327})
\end{proof}

\begin{thm}
On the metric in $\mathbb{K}$, for all $p,q,r$ in $\mathbb{K}$, those belows
are true.\newline
\begin{enumerate}
	\item[i.] $\left\langle rp,rq\right\rangle = N_{r}\langle p,q\rangle$,
	\item [ii.] $\left\langle pr,qr\right\rangle = N_{r}\langle p,q\rangle$,
	\item [iii.] $\langle pq,r\rangle = N_{r}\langle q,\bar{p}r\rangle$,
	\item [iv.] $\langle	pq,r\rangle = N_{r}\langle p,r\bar{q}\rangle$.
\end{enumerate} 
\end{thm}

\begin{proof} We prove the theorem by using Lemma 2.13 and Eq.(\ref{325}). We will prove the first equation. The proof of the other item has been left to the reader.
\begin{eqnarray*}
i.	\left\langle rp,rq\right\rangle=S(rp\overline{rq})=S(rp\overline{q}\bar{r})=S(\overline{q}\bar{r}rp)=N_{r}S(\overline{q}p)=N_{r}S(p\overline{q})=N_{r}\left\langle p,q\right\rangle.
\end{eqnarray*}%

\end{proof}

\section{Hamilton Operators and Matrices Associated with 3PGQs}

We can not think of quaternions independently of matrices. Real, split and
2PGQs have been also expressed by matrices and various applications have
been made on them. Some algebraic properties of Hamilton operators for both
2PGQs and dual quaternions in \cite{agr, jaf5, jaf6, olm} In this section we
associate 3PGQs with matrices.

\subsection{Obtaining the fundamental matrices}

In order to obtain the matrix $\mathcal{M}$, a 3PGQ $%
p=a_{0}+a_{1}e_{1}+a_{2}e_{2}+a_{3}e_{3}$ is multiplied from left side by $%
1,e_{1},e_{2},e_{3}$, 
\begin{eqnarray*}
\left( a_{0}+a_{1}e_{1}+a_{2}e_{2}+a_{3}e_{3}\right) 1
&=&a_{0}+a_{1}e_{1}+a_{2}e_{2}+a_{3}e_{3} \\
\left( a_{0}+a_{1}e_{1}+a_{2}e_{2}+a_{3}e_{3}\right) e_{1}
&=&a_{0}e_{1}+a_{1}e_{1}^{2}+a_{2}e_{2}e_{1}+a_{3}e_{3}e_{1} \\
&=&-\lambda _{1}\lambda _{2}a_{1}+a_{0}e_{1}+\lambda _{2}a_{3}e_{2}-\lambda
_{1}a_{2}e_{3} \\
\left( a_{0}+a_{1}e_{1}+a_{2}e_{2}+a_{3}e_{3}\right) e_{2}
&=&a_{0}e_{2}+a_{1}e_{1}e_{2}+a_{2}e_{2}^{2}+a_{3}e_{3}e_{2} \\
&=&-\lambda _{1}\lambda _{3}a_{2}-\lambda _{3}a_{3}e_{1}+a_{0}e_{2}+\lambda
_{1}a_{1}e_{3} \\
\left( a_{0}+a_{1}e_{1}+a_{2}e_{2}+a_{3}e_{3}\right) e_{3}
&=&a_{0}e_{3}+a_{1}e_{1}e_{3}+a_{2}e_{2}e_{3}+a_{3}e_{3}^{2} \\
&=&-\lambda _{2}\lambda _{3}a_{3}+\lambda _{3}a_{2}e_{1}-\lambda
_{2}a_{1}e_{2}+a_{0}e_{3}.
\end{eqnarray*}%
The coefficients of the equations in the above rows are the column elements
of the matrix $\mathcal{M}$: 
\begin{equation*}
\mathcal{M}=\left[ 
\begin{array}{cccc}
a_{0} & -\lambda _{1}\lambda _{2}a_{1} & -\lambda _{1}\lambda _{3}a_{2} & 
-\lambda _{2}\lambda _{3}a_{3} \\ 
a_{1} & a_{0} & -\lambda _{3}a_{3} & \lambda _{3}a_{2} \\ 
a_{2} & \lambda _{2}a_{3} & a_{0} & -\lambda _{2}a_{1} \\ 
a_{3} & -\lambda _{1}a_{2} & \lambda _{1}a_{1} & a_{0}%
\end{array}%
\right] .
\end{equation*}%
If we set $\lambda _{1}=1,$ $\lambda _{2}=\lambda ,$ $\lambda _{3}=\mu $,
then we found the fundamental matrix for 2PGQs. \newline
Taking $\lambda _{1}=1,$ $\lambda _{2}=1,$ $\lambda _{3}=-1$ in $\mathcal{M}$%
, we have the fundamental matrix for split quaternions.\newline
Similarly by setting $\lambda _{1}=1,$ $\lambda _{2}=1,$ $\lambda _{3}=1$,
Hamilton matrix is attained.\newline
Accordingly, multiplication of two 3PGQs as $%
p=a_{0}+a_{1}e_{1}+a_{2}e_{2}+a_{3}e_{3}$ and $%
q=b_{0}+b_{1}e_{1}+b_{2}e_{2}+b_{3}e_{3}$ can be obtained as follows: 
\begin{align*}
pq &=\left[ 
\begin{array}{cccc}
a_{0} & -\lambda _{1}\lambda _{2}a_{1} & -\lambda _{1}\lambda _{3}a_{2} & 
-\lambda _{2}\lambda _{3}a_{3} \\ 
a_{1} & a_{0} & -\lambda _{3}a_{3} & \lambda _{3}a_{2} \\ 
a_{2} & \lambda _{2}a_{3} & a_{0} & -\lambda _{2}a_{1} \\ 
a_{3} & -\lambda _{1}a_{2} & \lambda _{1}a_{1} & a_{0}%
\end{array}%
\right] \left[ 
\begin{array}{c}
b_{0} \\ 
b_{1} \\ 
b_{2} \\ 
b_{3}%
\end{array}%
\right]  \notag \\
&=\left[ 
\begin{array}{c}
a_{0}b_{0}-\lambda _{1}\lambda _{2}a_{1}b_{1}-\lambda _{1}\lambda
_{3}a_{2}b_{2}-\lambda _{2}\lambda _{3}a_{3}b_{3} \\ 
a_{0}b_{1}+a_{1}b_{0}+\lambda _{3}a_{2}b_{3}-\lambda _{3}a_{3}b_{2} \\ 
a_{0}b_{2}+a_{2}b_{0}-\lambda _{2}a_{1}b_{3}+\lambda _{2}a_{3}b_{1} \\ 
a_{0}b_{3}+b_{0}a_{3}+\lambda _{1}a_{1}b_{2}-\lambda _{1}a_{2}b_{1}%
\end{array}%
\right] \allowbreak
\end{align*}

\begin{thm}
The 3PGQ ring $\mathbb{K}$ is isomorphic to a subring of the ring $\mathbb{M}%
_{4}\left( \mathbb{R}\right)$.
\end{thm}

\begin{proof}
Let us define the mapping	$\phi :(\mathbb{K},+,\times )\rightarrow (%
	\mathbb{M}_{4}\left( \mathbb{R}\right) ,\oplus ,\otimes )$, where
	\begin{equation*}
	\phi \left( a_{0}+a_{1}e_{1}+a_{2}e_{2}+a_{3}e_{3}\right) \rightarrow \left[ 
	\begin{array}{cccc}
	a_{0} & -\lambda _{1}\lambda _{2}a_{1} & -\lambda _{1}\lambda _{3}a_{2} & 
	-\lambda _{2}\lambda _{3}a_{3} \\ 
	a_{1} & a_{0} & -\lambda _{3}a_{3} & \lambda _{3}a_{2} \\ 
	a_{2} & \lambda _{2}a_{3} & a_{0} & -\lambda _{2}a_{1} \\ 
	a_{3} & -\lambda _{1}a_{2} & \lambda _{1}a_{1} & a_{0}%
	\end{array}%
	\right]. 
	\end{equation*}%
	Let us prove that the mapping $\phi$ is a ring isomorphism. Taking into account the addition and multiplication operations achieved for the 3PGQ, it can easily be shown by the reader these equalities:
\begin{eqnarray*}
\phi (p+q) &=& \phi (p)\oplus \phi (q) \\
\phi (pq) &=&\phi (p)\otimes\phi (q)
	\end{eqnarray*}%
Now let us show that $\phi$ is bijective. Since
	\begin{eqnarray*}
		Ker\phi &=&\left\{ p:\phi \left( p\right) =0\right\} \\
		&=&\left\{ p:0\mathcal{=}\left[ 
		\begin{array}{cccc}
			0 & 0 & 0 & 0 \\ 
			0 & 0 & 0 & 0 \\ 
			0 & 0 & 0 & 0 \\ 
			0 & 0 & 0 & 0%
		\end{array}%
		\right] \right\} \\
		&=&\left\{ 0\right\}
	\end{eqnarray*}%
	$\phi$ is one-to-one.
	\begin{eqnarray*}
		\phi \left( \mathbb{K}\right) &=&\left\{ \phi \left( p\right)
		:p\in \mathbb{K}\right\} \\
	&=&\left\{ \left[ 
		\begin{array}{cccc}
			a_{0} & -\lambda _{1}\lambda _{2}a_{1} & -\lambda _{1}\lambda _{3}a_{2} & 
			-\lambda _{2}\lambda _{3}a_{3} \\ 
			a_{1} & a_{0} & -\lambda _{3}a_{3} & \lambda _{3}a_{2} \\ 
			a_{2} & \lambda _{2}a_{3} & a_{0} & -\lambda _{2}a_{1} \\ 
			a_{3} & -\lambda _{1}a_{2} & \lambda _{1}a_{1} & a_{0}%
		\end{array}%
		\right] :a_{i}\in \mathbb{R}\right\}.
	\end{eqnarray*}%
	If we take the restriction
	\begin{equation*}
	\phi :\mathbb{K}\rightarrow \phi \left( \mathbb{K}\right) \subset \mathbb{M}_{4}\left( \mathbb{R}\right)
	\end{equation*}%
	because of our choice of the value set, the mapping $\phi$ is bijective.
\end{proof}In order to obtain the matrix $\mathcal{N}$, any 3PGQ $%
p=a_{0}+a_{1}e_{1}+a_{2}e_{2}+a_{3}e_{3}$ is multiplied from right side by $%
1,e_{1},e_{2},e_{3}$. Similarly to production of $\mathcal{M}$, the $%
\mathcal{N}$ matrix is produced: 
\begin{equation*}
\mathcal{N}=\left[ 
\begin{array}{cccc}
a_{0} & -\lambda _{1}\lambda _{2}a_{1} & -\lambda _{1}\lambda _{3}a_{2} & 
-\lambda _{2}\lambda _{3}a_{3} \\ 
a_{1} & a_{0} & \lambda _{3}a_{3} & -\lambda _{3}a_{2} \\ 
a_{2} & -\lambda _{2}a_{3} & a_{0} & \lambda _{2}a_{1} \\ 
a_{3} & \lambda _{1}a_{2} & -\lambda _{1}a_{1} & a_{0}%
\end{array}%
\right]
\end{equation*}%
As a result, there are two fundamental matrices that give the algebra of
3PGQs: $\mathcal{M}$\ and $\mathcal{N}$. Throughout the article, since all
the operations with $\mathcal{M}$\ and $\mathcal{N}$ matrices will proceed
in a similar way, we will only give definition theorems and explanations for
only the matrix $\mathcal{M}$ and consider that for the matrix $\mathcal{N}$
can be done in a similar way.

\subsection{Obtaining the multiplication table with the help of fundamental
matrices}

From the matrix $\mathcal{M}$ that we have obtained in the previous section,
we achieve the base elements $e_{0},e_{1},e_{2},e_{3}$ as follows%
\begin{equation*}
{\small 
\begin{array}{cc}
e_{0}=1\leftrightarrow \left[ 
\begin{array}{cccc}
1 & 0 & 0 & 0 \\ 
0 & 1 & 0 & 0 \\ 
0 & 0 & 1 & 0 \\ 
0 & 0 & 0 & 1%
\end{array}%
\right] =E_{0}=I_{4}, & e_{1}\leftrightarrow \left[ 
\begin{array}{cccc}
0 & -\lambda _{1}\lambda _{2} & 0 & 0 \\ 
1 & 0 & 0 & 0 \\ 
0 & 0 & 0 & -\lambda _{2} \\ 
0 & 0 & \lambda _{1} & 0%
\end{array}%
\right] =E_{1}, \\ 
e_{2}\leftrightarrow \left[ 
\begin{array}{cccc}
0 & 0 & -\lambda _{1}\lambda _{3} & 0 \\ 
0 & 0 & 0 & \lambda _{3} \\ 
1 & 0 & 0 & 0 \\ 
0 & -\lambda _{1} & 0 & 0%
\end{array}%
\,\right] =E_{2}, & e_{3}\leftrightarrow \left[ 
\begin{array}{cccc}
0 & 0 & 0 & -\lambda _{2}\lambda _{3} \\ 
0 & 0 & -\lambda _{3} & 0 \\ 
0 & \lambda _{2} & 0 & 0 \\ 
1 & 0 & 0 & 0%
\end{array}%
\right] =E_{3}.%
\end{array}%
}
\end{equation*}
where $\left\{ E_{0},E_{1},E_{2},E_{3}\right\} $ is the set of base matrices
which corresponding to the base elements $1,e_{1},e_{2},e_{3}$. Accordingly,
multiplying these matrices with each other yields the followings: 
\begin{eqnarray*}
e_{1}^{2} &\leftrightarrow &-\lambda _{1}\lambda _{2}I_{4},\text{ \ \ \ }%
e_{2}^{2}\leftrightarrow -\lambda _{1}\lambda _{3}I_{4},\text{ \ \ }%
e_{3}^{2}\leftrightarrow -\lambda _{2}\lambda _{3}I_{4}, \\
e_{1}e_{2} &\leftrightarrow &\lambda _{1}E_{3},\text{ \ \ \ \ }%
e_{2}e_{1}\leftrightarrow -\lambda _{1}E_{3},\text{ \ \ }e_{2}e_{3}%
\leftrightarrow \lambda _{3}E_{1}, \\
e_{3}e_{2} &\leftrightarrow &-\lambda _{3}E_{1},\text{ \ \ }%
e_{1}e_{3}\leftrightarrow -\lambda _{2}E_{2},\text{ \ \ }e_{3}e_{1}%
\leftrightarrow \lambda _{2}E_{2}, \\
e_{1}e_{2}e_{3} &\leftrightarrow &-\lambda _{1}\lambda _{2}\lambda _{3}I_{4},%
\text{ \ \ }e_{2}e_{3}e_{1}\leftrightarrow -\lambda _{1}\lambda _{2}\lambda
_{3}I_{4}, \\
e_{3}e_{1}e_{2} &\leftrightarrow &-\lambda _{1}\lambda _{2}\lambda _{3}I_{4},%
\text{ \ \ }e_{1}e_{3}e_{2}\leftrightarrow \lambda _{1}\lambda _{2}\lambda
_{3}I_{4}, \\
e_{2}e_{1}e_{3} &\leftrightarrow &\lambda _{1}\lambda _{2}\lambda _{3}I_{4},%
\text{ \ \ }e_{3}e_{2}e_{1}\leftrightarrow \lambda _{1}\lambda _{2}\lambda
_{3}I_{4}
\end{eqnarray*}%
which gives us the multiplication table in Definition 2.1.
\subsection{Determinant, characteristic polynomial, characteristic equation, eigenvalues and eigenvectors of the matrix $\mathcal{M}$}

Determinant of matrix $\mathcal{M}$ is calculated as follows: 
\begin{align*}
\left\vert \mathcal{M}\right\vert &\mathcal{=}\left\vert 
\begin{array}{cccc}
a_{0} & -\lambda _{1}\lambda _{2}a_{1} & -\lambda _{1}\lambda _{3}a_{2} & 
-\lambda _{2}\lambda _{3}a_{3} \\ 
a_{1} & a_{0} & -\lambda _{3}a_{3} & \lambda _{3}a_{2} \\ 
a_{2} & \lambda _{2}a_{3} & a_{0} & -\lambda _{2}a_{1} \\ 
a_{3} & -\lambda _{1}a_{2} & \lambda _{1}a_{1} & a_{0}%
\end{array}%
\right\vert =\left( N_{p}\right) ^{2},  \label{421}
\end{align*}%
where $p=a_{0}+a_{1}e_{1}+a_{2}e_{2}+a_{3}e_{3}$.\newline
Characteristic polynomial of the matrix $\mathcal{M}$ is 
\begin{equation*}
P_{\mathcal{M}}\left( t\right) \allowbreak =\left(
t^{2}-2ta_{0}+a_{0}^{2}+\lambda _{1}\lambda _{2}a_{1}^{2}+\lambda
_{1}\lambda _{3}a_{2}^{2}+\lambda _{2}\lambda _{3}a_{3}^{2}\right)
^{2}\allowbreak.
\end{equation*}
Characteristic equation of the matrix $\mathcal{M}$ is 
\begin{align*}
\det \left( \mathcal{M-}tI_{4}\right) &=0  \notag \\
0 &=\left\vert 
\begin{array}{cccc}
a_{0}-t & -\lambda _{1}\lambda _{2}a_{1} & -\lambda _{1}\lambda _{3}a_{2} & 
-\lambda _{2}\lambda _{3}a_{3} \\ 
a_{1} & a_{0}-t & -\lambda _{3}a_{3} & \lambda _{3}a_{2} \\ 
a_{2} & \lambda _{2}a_{3} & a_{0}-t & -\lambda _{2}a_{1} \\ 
a_{3} & -\lambda _{1}a_{2} & \lambda _{1}a_{1} & a_{0}-t%
\end{array}%
\right\vert  \notag \\
0 &=\left( t^{2}-2ta_{0}+a_{0}^{2}+\lambda _{1}\lambda _{2}a_{1}^{2}+\lambda
_{1}\lambda _{3}a_{2}^{2}+\lambda _{2}\lambda _{3}a_{3}^{2}\right)
^{2}\allowbreak.
\end{align*}%
The four eigenvalues are both coincident in pairs and each other's
conjugate: 
\begin{eqnarray*}
t_{1,2} &=&a_{0}+\sqrt{-\lambda _{1}\lambda _{2}a_{1}^{2}-\lambda
_{1}\lambda _{3}a_{2}^{2}-\lambda _{2}\lambda _{3}a_{3}^{2}} \\
t_{3,4} &=&a_{0}-\sqrt{-\lambda _{1}\lambda _{2}a_{1}^{2}-\lambda
_{1}\lambda _{3}a_{2}^{2}-\lambda _{2}\lambda _{3}a_{3}^{2}}\allowbreak .
\end{eqnarray*}%
Multiplication of the eigenvalues is achieved as 
\begin{equation*}
a_{0}^{2}+\lambda _{1}\lambda _{2}a_{1}^{2}+\lambda _{1}\lambda
_{3}a_{2}^{2}+\lambda _{2}\lambda _{3}a_{3}^{2}=N(q).
\end{equation*}

Also there are two eigenvectors corresponding to the eigenvalue \linebreak $%
a_{0}+\sqrt{-\lambda _{1}\lambda _{2}a_{1}^{2}-\lambda _{1}\lambda
_{3}a_{2}^{2}-\lambda _{2}\lambda _{3}a_{3}^{2}}$ and these are
\begin{eqnarray*}
&&\left[ 
\begin{array}{c}
\frac{\lambda _{1}a_{2}\sqrt{-\lambda _{1}\lambda _{2}a_{1}^{2}-\lambda
_{1}\lambda _{3}a_{2}^{2}-\lambda _{2}\lambda _{3}a_{3}^{2}}-\lambda
_{1}\lambda _{2}a_{1}a_{3}}{\lambda _{1}a_{2}^{2}+\lambda _{2}a_{3}^{2}} \\ 
\frac{a_{3}\sqrt{-\lambda _{1}\lambda _{2}a_{1}^{2}-\lambda _{1}\lambda
_{3}a_{2}^{2}-\lambda _{2}\lambda _{3}a_{3}^{2}}+\lambda _{1}a_{1}a_{2}}{
\lambda _{1}a_{2}^{2}+\lambda _{2}a_{3}^{2}} \\ 
1 \\ 
0%
\end{array}
\right], \\
&&\left[ 
\begin{array}{c}
\frac{\lambda _{2}a_{3}\sqrt{-\lambda _{1}\lambda _{2}a_{1}^{2}-\lambda
_{1}\lambda _{3}a_{2}^{2}-\lambda _{2}\lambda _{3}a_{3}^{2}}+\lambda
_{1}\lambda _{2}a_{1}a_{2}}{\lambda _{1}a_{2}^{2}+\lambda _{2}a_{3}^{2}} \\ 
-\frac{a_{2}\sqrt{-\lambda _{1}\lambda _{2}a_{1}^{2}-\lambda _{1}\lambda
_{3}a_{2}^{2}-\lambda _{2}\lambda _{3}a_{3}^{2}}-\lambda _{2}a_{1}a_{3}}{
\lambda _{1}a_{2}^{2}+\lambda _{2}a_{3}^{2}} \\ 
0 \\ 
1%
\end{array}
\right] .
\end{eqnarray*}%
The eigenvectors corresponding to the eigenvalue \linebreak $a_{0}-\sqrt{%
-\lambda _{1}\lambda _{2}a_{1}^{2}-\lambda _{1}\lambda _{3}a_{2}^{2}-\lambda
_{2}\lambda _{3}a_{3}^{2}}$ are

\begin{eqnarray*}
&&\left[ 
\begin{array}{c}
-\frac{\lambda _{1}a_{2}\sqrt{-\lambda _{1}\lambda _{2}a_{1}^{2}-\lambda
_{1}\lambda _{3}a_{2}^{2}-\lambda _{2}\lambda _{3}a_{3}^{2}}+\lambda
_{1}\lambda _{2}a_{1}a_{3}}{\lambda _{1}a_{2}^{2}+\lambda _{2}a_{3}^{2}} \\ 
-\frac{a_{3}\sqrt{-\lambda _{1}\lambda _{2}a_{1}^{2}-\lambda _{1}\lambda
_{3}a_{2}^{2}-\lambda _{2}\lambda _{3}a_{3}^{2}}-\lambda _{1}a_{1}a_{2}}{
\lambda _{1}a_{2}^{2}+\lambda _{2}a_{3}^{2}} \\ 
1 \\ 
0%
\end{array}
\right],
\end{eqnarray*}
\begin{eqnarray*}
&&\left[ 
\begin{array}{c}
-\frac{\lambda _{2}a_{3}\sqrt{-\lambda _{1}\lambda _{2}a_{1}^{2}-\lambda
_{1}\lambda _{3}a_{2}^{2}-\lambda _{2}\lambda _{3}a_{3}^{2}}-\lambda
_{1}\lambda _{2}a_{1}a_{2}}{\lambda _{1}a_{2}^{2}+\lambda _{2}a_{3}^{2}} \\ 
\frac{a_{2}\sqrt{-\lambda _{1}\lambda _{2}a_{1}^{2}-\lambda _{1}\lambda
_{3}a_{2}^{2}-\lambda _{2}\lambda _{3}a_{3}^{2}}+\lambda _{2}a_{1}a_{3}}{
\lambda _{1}a_{2}^{2}+\lambda _{2}a_{3}^{2}} \\ 
0 \\ 
1%
\end{array}
\right].
\end{eqnarray*}
\section{Polar Representation, De Moivre's and Euler's Formulas for 3PGQs}
Euler's and De Moivre's formulas in complex number are generalized for
Hamilton quaternions in \cite{cho}. It has also been studied for split and
dual quaternions in \cite{kab, ozd}. Recently, De Moivre's and Euler's formulas
have been derived for matrices associated with real, dual quaternions \cite%
{jaf2, jaf3}. In generalized quaternion algebra, De Moivre's and Euler's
formulas are studied in \cite{jaf4}. In this section, the polar
representation of 3PGQs is studied. And the polar matrix representation of
the fundamental matrix $\mathcal{M}$ is created and De Moivre's and Euler's
formulas are composed for 3PGQs and the matrix $\mathcal{M}$.

\subsection{Polar representation of 3PGQs and the matrix $\mathcal{M}$}

We can associate an angle $\theta $ with a 3PGQ $%
p=a_{0}+a_{1}e_{1}+a_{2}e_{2}+a_{3}e_{3}$ as 
\begin{equation*}
\cos \theta =\frac{a_{0}}{\sqrt{N(p)}} \text{ \ and }\sin \theta =\frac{\sqrt{
\lambda _{1}\lambda _{2}a_{1}^{2}+\lambda _{1}\lambda _{3}a_{2}^{2}+\lambda
_{2}\lambda _{3}a_{3}^{2}}}{\sqrt{N(p)}}.
\end{equation*}

\begin{defn}
Any 3PGQ $p$ can be written in the polar form as the following: 
\begin{equation}
p=\sqrt{N(p)}\left( \cos \theta +\hat{p}\sin \theta \right)  \label{6}
\end{equation}%
where
\begin{equation*}
\hat{p}=\dfrac{\left( a_{1},a_{2},a_{3}\right) }{\sqrt{\lambda _{1}\lambda
_{2}a_{1}^{2}+\lambda _{1}\lambda _{3}a_{2}^{2}+\lambda _{2}\lambda
_{3}a_{3}^{2}}}
\end{equation*}%
is a 3-parameter generalized unit vector (3PGUV).
We will use $\hat{p}=\left( p_{1},p_{2},p_{3}\right) $ in order to
be more simple and short, where 
\begin{eqnarray*}
p_{1} &=&\dfrac{a_{1}}{\sqrt{\lambda _{1}\lambda _{2}a_{1}^{2}+\lambda
_{1}\lambda _{3}a_{2}^{2}+\lambda _{2}\lambda _{3}a_{3}^{2}}}, \\
p_{2} &=&\dfrac{a_{2}}{\sqrt{\lambda _{1}\lambda _{2}a_{1}^{2}+\lambda
_{1}\lambda _{3}a_{2}^{2}+\lambda _{2}\lambda _{3}a_{3}^{2}}}, \\
p_{3} &=&\dfrac{a_{3}}{\sqrt{\lambda _{1}\lambda _{2}a_{1}^{2}+\lambda
_{1}\lambda _{3}a_{2}^{2}+\lambda _{2}\lambda _{3}a_{3}^{2}}}.
\end{eqnarray*}%
Indeed, the form of $\hat{p}$ into Eq.(\ref{6}) is shown as the following: 
\begin{eqnarray*}
p &=&a_{0}+a_{1}e_{1}+a_{2}e_{2}+a_{3}e_{3} \\
&=&\sqrt{N(p)}\left( \frac{a_{0}}{\sqrt{N(p)}}+\frac{1}{\sqrt{N(p)}}\left(
a_{1}e_{1}+a_{2}e_{2}+a_{3}e_{3}\right) \right) \\
&=&\sqrt{N(p)}(\frac{a_{0}}{\sqrt{N(p)}}+\frac{\sqrt{\lambda _{1}\lambda
_{2}a_{1}^{2}+\lambda _{1}\lambda _{3}a_{2}^{2}+\lambda _{2}\lambda
_{3}a_{3}^{2}}}{\sqrt{N(p)}}. \\
&&(\frac{a_{1}}{\sqrt{\lambda _{1}\lambda _{2}a_{1}^{2}+\lambda _{1}\lambda
_{3}a_{2}^{2}+\lambda _{2}\lambda _{3}a_{3}^{2}}}e_{1}+\frac{a_{2}}{\sqrt{
\lambda _{1}\lambda _{2}a_{1}^{2}+\lambda _{1}\lambda _{3}a_{2}^{2}+\lambda
_{2}\lambda _{3}a_{3}^{2}}}e_{2} \\
&&+\frac{a_{3}}{\sqrt{\lambda _{1}\lambda _{2}a_{1}^{2}+\lambda _{1}\lambda
_{3}a_{2}^{2}+\lambda _{2}\lambda _{3}a_{3}^{2}}}e_{3})) \\
&=&\sqrt{N(p)}\left( \cos \theta +\left( p_{1},p_{2},p_{3}\right) \sin
\theta .\right) \\
&=&\sqrt{N(p)}\left( \cos \theta +\hat{p}\sin \theta \right).
\end{eqnarray*}
\subsection{Polar representation of the matrix $\mathcal{M}$} 
Let $p$ be a 3PGUQ. We can write 
\begin{eqnarray*}
p &=&a_{0}+a_{1}e_{1}+a_{2}e_{2}+a_{3}e_{3} \\
&=&\cos \theta +\hat{p}\sin \theta \\
&=&\cos \theta +\left( p_{1},p_{2},p_{3}\right) \sin \theta \\
&=&\cos \theta +p_{1}\sin \theta +p_{2}\sin \theta +p_{3}\sin \theta \\
&=&\left( \cos \theta ,p_{1}\sin \theta ,p_{2}\sin \theta ,p_{3}\sin \theta
\right)
\end{eqnarray*}%
and polar form of the matrix $\mathcal{M}$ is obtained as follows:
\begin{eqnarray*}
\mathcal{M} &=&\left[ 
\begin{array}{cccc}
\cos \theta & -\lambda _{1}\lambda _{2}p_{1}\sin \theta & -\lambda
_{1}\lambda _{3}p_{2}\sin \theta & -\lambda _{2}\lambda _{3}p_{3}\sin \theta
\\ 
p_{1}\sin \theta & \cos \theta & -\lambda _{3}p_{3}\sin \theta & \lambda
_{3}p_{2}\sin \theta \\ 
p_{2}\sin \theta & \lambda _{2}p_{3}\sin \theta & \cos \theta & -\lambda
_{2}p_{1}\sin \theta \\ 
p_{3}\sin \theta & -\lambda _{1}p_{2}\sin \theta & \lambda _{1}p_{1}\sin
\theta & \cos \theta%
\end{array}
\right].
\end{eqnarray*}
\end{defn}

\subsection{De Moivre's formula for 3PGQs}

Let us represent the set of 3PGUQs as $S_{\mathbb{K}}$ and the set of
3-parameter generalized unit vectors (3PGUVs) as $S_{\mathbb{K}}^{2}$. Namely%
\newline
$S_{\mathbb{K}}=\left\{ p\in \mathbb{K}:N_{p}=1\right\}$,%
\newline
$S_{\mathbb{K}}^{2}=\left\{ h \in \text{Im}(\mathbb{K}) :N_{h}=1\right\}$.

\begin{lem}
If \textit{\ }$v\in S_{\mathbb{K}}^{2}$, then 
\begin{equation*}
\left( \cos \alpha +v\sin \alpha \right) \left( \cos \beta +v\sin \beta
\right) =\cos \left( \alpha +\beta \right) +v\sin \left( \alpha +\beta
\right) .
\end{equation*}
\end{lem}

\begin{proof}
Proof is made similar to the proof in \cite{mer}.
\end{proof}

\begin{thm}
For $p\in S_{\mathbb{K}}$, if $p=\cos \theta +\hat{p}\sin \theta $, then 
\begin{equation*}
p^{n}=\left( \cos \theta +\hat{p}\sin \theta \right) ^{n}=\cos \left(
n\theta \right) +\hat{p}\sin \left( n\theta \right) .
\end{equation*}
\end{thm}

\begin{proof}
The theorem is easily proved by using Lemma 4.2 and the induction method, similar to the proof in \cite{mer} 
\end{proof}

\subsection{De Moivre's formula for matrices associated with 3PGQs}

We will obtain De Moivre's formula for the matrices corresponding to the
3PGQ $p$. Let $p=\cos\alpha +\hat{p}\sin \alpha $ be polar representation of
a 3PGQ, where $\hat{p}$ is a 3PGUQ.

\begin{lem}
\begin{eqnarray*}
P &=&\left[ 
\begin{array}{cccc}
\cos \alpha & -\lambda _{1}\lambda _{2}p_{1}\sin \alpha & -\lambda
_{1}\lambda _{3}p_{2}\sin \alpha & -\lambda _{2}\lambda _{3}p_{3}\sin \alpha
\\ 
p_{1}\sin \alpha & \cos \alpha & -\lambda _{3}p_{3}\sin \alpha & \lambda
_{3}p_{2}\sin \alpha \\ 
p_{2}\sin \alpha & \lambda _{2}p_{3}\sin \alpha & \cos \alpha & -\lambda
_{2}p_{1}\sin \alpha \\ 
p_{3}\sin \alpha & -\lambda _{1}p_{2}\sin \alpha & \lambda _{1}p_{1}\sin
\alpha & \cos \alpha%
\end{array}
\right] \\
Q &=&\left[ 
\begin{array}{cccc}
\cos \beta & -\lambda _{1}\lambda _{2}p_{1}\sin \beta & -\lambda _{1}\lambda
_{3}p_{2}\sin \beta & -\lambda _{2}\lambda _{3}p_{3}\sin \beta \\ 
p_{1}\sin \beta & \cos \beta & -\lambda _{3}p_{3}\sin \beta & \lambda
_{3}p_{2}\sin \beta \\ 
p_{2}\sin \beta & \lambda _{2}p_{3}\sin \beta & \cos \beta & -\lambda
_{2}p_{1}\sin \beta \\ 
p_{3}\sin \beta & -\lambda _{1}p_{2}\sin \beta & \lambda _{1}p_{1}\sin \alpha
& \cos \beta%
\end{array}
\right]
\end{eqnarray*}%
the matrix $PQ$ $\allowbreak$ is achieved as 
\begin{equation*}
{\small \left[ 
\begin{array}{cccc}
\cos \text{(}\alpha \text{+}\beta \text{)} & -\lambda _{1}\lambda
_{2}p_{1}\sin \left( \alpha \text{+}\beta \right) & -\lambda _{1}\lambda
_{3}p_{2}\sin \left( \alpha \text{+}\beta \right) & -\lambda _{2}\lambda
_{3}p_{3}\sin \left( \alpha \text{+}\beta \right) \\ 
p_{1}\sin \left( \alpha \text{+}\beta \right) & \cos (\alpha \text{+}\beta )
& -\lambda _{3}p_{3}\sin \left( \alpha \text{+}\beta \right) & \lambda
_{3}p_{2}\sin \left( \alpha \text{+}\beta \right) \\ 
p_{2}\sin \left( \alpha \text{+}\beta \right) & \lambda _{2}p_{3}\sin \left(
\alpha \text{+}\beta \right) & \cos (\alpha \text{+}\beta ) & -\lambda
_{2}p_{1}\sin \left( \alpha \text{+}\beta \right) \\ 
p_{3}\sin \left( \alpha \text{+}\beta \right) & -\lambda _{1}p_{2}\sin
\left( \alpha \text{+}\beta \right) & \lambda _{1}p_{1}\sin \left( \alpha 
\text{+}\beta \right) & \cos (\alpha \text{+}\beta )%
\end{array}%
\right]}.
\end{equation*}
\end{lem}

\begin{proof}	
	Let $PQ=\left[ a_{ij}\right]_{4\times 4} $.	
	\begin{eqnarray*}
		a_{11} &=&a_{22}=a_{33}=a_{44} \\
		&=&\cos \alpha \cos \beta -\lambda _{1}\lambda _{2}p_{1}^{2}\sin \alpha \sin
		\beta -\lambda _{1}\lambda _{3}p_{2}^{2}\sin \alpha \sin \beta \\ &-&\lambda
		_{2}\lambda _{3}p_{3}^{2}\sin \alpha \sin \beta  \\
		&=&\cos \alpha \cos \beta -\left( \lambda _{1}\lambda _{2}p_{1}^{2}+\lambda
		_{1}\lambda _{3}p_{2}^{2}+\lambda _{2}\lambda _{3}p_{3}^{2}\right) \sin
		\alpha \sin \beta  \\
		&=&\cos \alpha \cos \beta -\sin \alpha \sin \beta =\cos (\alpha +\beta )
	\end{eqnarray*}%
	and%
	\begin{eqnarray*}
		a_{12} &=&-\lambda _{1}\lambda _{2}p_{1}\cos \alpha \sin \beta -\lambda
		_{1}\lambda _{2}p_{1}\sin \alpha \cos \beta  \\
		&&-\lambda _{1}\lambda _{2}\lambda _{3}p_{2}p_{3}\sin \alpha \sin \beta
		+\lambda _{1}\lambda _{2}\lambda _{3}p_{2}p_{3}\sin \alpha \sin \beta  \\
		&=&-\lambda _{1}\lambda _{2}p_{1}\left( \cos \alpha \sin \beta +\sin \alpha
		\cos \beta \right)  \\
		&&-\lambda _{1}\lambda _{2}\lambda _{3}p_{2}p_{3}\left( \sin \alpha \sin
		\beta -\sin \alpha \sin \beta \right)  \\
		\qquad  &=&-\lambda _{1}\lambda _{2}p_{1}\sin \left( \alpha +\beta \right) 
	\end{eqnarray*}
	Similarly, necessary calculations are made and the other elements are attained. 
\end{proof}
\begin{thm}
For any integer $n$, if 
\begin{equation*}
P=\left[ 
\begin{array}{cccc}
\cos \alpha & -\lambda _{1}\lambda _{2}p_{1}\sin \alpha & -\lambda
_{1}\lambda _{3}p_{2}\sin \alpha & -\lambda _{2}\lambda _{3}p_{3}\sin \alpha
\\ 
p_{1}\sin \alpha & \cos \alpha & -\lambda _{3}p_{3}\sin \alpha & \lambda
_{3}p_{2}\sin \alpha \\ 
p_{2}\sin \alpha & \lambda _{2}p_{3}\sin \alpha & \cos \alpha & -\lambda
_{2}p_{1}\sin \alpha \\ 
p_{3}\sin \alpha & -\lambda _{1}p_{2}\sin \alpha & \lambda _{1}p_{1}\sin
\alpha & \cos \alpha%
\end{array}%
\right],
\end{equation*}%
then the $n$th power of matrix $P$ is obtained as 
\begin{equation*}
\left[ 
\begin{array}{cccc}
\cos \left( n\alpha \right) & -\lambda _{1}\lambda _{2}p_{1}\sin \left(
n\alpha \right) & -\lambda _{1}\lambda _{3}p_{2}\sin \left( n\alpha \right)
& -\lambda _{2}\lambda _{3}p_{3}\sin \left( n\alpha \right) \\ 
p_{1}\sin \left( n\alpha \right) & \cos \left( n\alpha \right) & -\lambda
_{3}p_{3}\sin \left( n\alpha \right) & \lambda _{3}p_{2}\sin \left( n\alpha
\right) \\ 
p_{2}\sin \left( n\alpha \right) & \lambda _{2}p_{3}\sin \left( n\alpha
\right) & \cos \left( n\alpha \right) & -\lambda _{2}p_{1}\sin \left(
n\alpha \right) \\ 
p_{3}\sin \left( n\alpha \right) & -\lambda _{1}p_{2}\sin \left( n\alpha
\right) & \lambda _{1}p_{1}\sin \left( n\alpha \right) & \cos \left( n\alpha
\right)%
\end{array}%
\right].
\end{equation*}
\end{thm}

\begin{proof}
We can prove this by means of induction method. First, let us show correctness of the theorem for $n\geq 2$.
For $n=2$, by using Lemma 5.4 and replacing the $Q$ matrix with the $P$ matrix, we obtain 
\begin{equation*} \small
P^{2}=\left[ 
\begin{array}{cccc}
\cos \left( 2\alpha \right) & -\lambda _{1}\lambda _{2}p_{1}\sin \left(
2\alpha \right) & -\lambda _{1}\lambda _{3}p_{2}\sin \left( 2\alpha \right)
& -\lambda _{2}\lambda _{3}p_{3}\sin \left( 2\alpha \right) \\ 
p_{1}\sin \left( 2\alpha \right) & \cos \left( 2\alpha \right) & -\lambda
_{3}p_{3}\sin \left( 2\alpha \right) & \lambda _{3}p_{2}\sin \left( 2\alpha
\right) \\ 
p_{2}\sin \left( 2\alpha \right) & \lambda _{2}p_{3}\sin \left( 2\alpha
\right) & \cos \left( 2\alpha \right) & -\lambda _{2}p_{1}\sin \left(
2\alpha \right) \\ 
p_{3}\sin \left( 2\alpha \right) & -\lambda _{1}p_{2}\sin \left( 2\alpha
\right) & \lambda _{1}p_{1}\sin \left( 2\alpha \right) & \cos \left( 2\alpha
\right)%
\end{array}%
\right].
\end{equation*}%
For $n=k$, let it be correct. For $n=k+1$, by using
$P^{k+1}=P^{k}P$ and Lemma 5.4, $P^{k+1}$ is obtained as
\begin{eqnarray*}
	 \left[ 
	\begin{array}{ccc}
		\cos \left( \left( k+1\right) \alpha \right) & -\lambda _{1}\lambda
		_{2}p_{1}\sin \left( \left( k+1\right) \alpha \right) & -\lambda _{1}\lambda
		_{3}p_{2}\sin \left( \left( k+1\right) \alpha \right) \\ 
		p_{1}\sin \left( \left( k+1\right) \alpha \right) & \cos \left( \left(
		k+1\right) \alpha \right) & -\lambda _{3}p_{3}\sin \left( \left( k+1\right)
		\alpha \right) \\ 
		p_{2}\sin \left( \left( k+1\right) \alpha \right) & \lambda _{2}p_{3}\sin
		\left( \left( k+1\right) \alpha \right) & \cos \left( \left( k+1\right)
		\alpha \right) \\ 
		p_{3}\sin \left( \left( k+1\right) \alpha \right) & -\lambda _{1}p_{2}\sin
		\left( \left( k+1\right) \alpha \right) & \lambda _{1}p_{1}\sin \left(
		\left( k+1\right) \alpha \right)%
	\end{array}%
	\right. \\
	\hspace*{7cm}\left. 
	\begin{array}{c}
		-\lambda _{2}\lambda _{3}p_{3}\sin \left( \left( k+1\right) \alpha \right)
		\\ 
		\lambda _{3}p_{2}\sin \left( \left( k+1\right) \alpha \right) \\ 
		-\lambda _{2}p_{1}\sin \left( \left( k+1\right) \alpha \right) \\ 
		\cos \left( \left( k+1\right) \alpha \right)%
	\end{array}%
	\right]
\end{eqnarray*}%
We find the matrix $P^{-1}$ by calculation of inverse matrix  as
\begin{equation*}
P^{-1}=\left[ 
\begin{array}{cccc}
\cos \alpha & \lambda _{1}\lambda _{2}p_{1}\sin \alpha & \lambda _{1}\lambda
_{3}p_{2}\sin \alpha & \lambda _{2}\lambda _{3}p_{3}\sin \alpha \\ 
-p_{1}\sin \alpha & \cos \alpha & \lambda _{3}p_{3}\sin \alpha & -\lambda
_{3}p_{2}\sin \alpha \\ 
-p_{2}\sin \alpha & -\lambda _{2}p_{3}\sin \alpha & \cos \alpha & \lambda
_{2}p_{1}\sin \alpha \\ 
-p_{3}\sin \alpha & \lambda _{1}p_{2}\sin \alpha & -\lambda _{1}p_{1}\sin
\alpha & \cos \alpha%
\end{array}%
\right].
\end{equation*}%
Since cosine is even function and sine is odd function, the matrix $P^{-1}$ is written as
\begin{equation*} \small
P^{\text{-1}}=\left[ 
\begin{array}{cccc}
\cos \left( -\alpha \right) & -\lambda _{1}\lambda _{2}p_{1}\sin \left(
-\alpha \right) & -\lambda _{1}\lambda _{3}p_{2}\sin \left( -\alpha \right)
& -\lambda _{2}\lambda _{3}p_{3}\sin \left( -\alpha \right) \\ 
p_{1}\sin \left( -\alpha \right) & \cos \left( -\alpha \right) & -\lambda
_{3}p_{3}\sin \left( -\alpha \right) & \lambda _{3}p_{2}\sin \left( -\alpha
\right) \\ 
p_{2}\sin \left( -\alpha \right) & \lambda _{2}p_{3}\sin \left( -\alpha
\right) & \cos \left( -\alpha \right) & -\lambda _{2}p_{1}\sin \left(
-\alpha \right) \\ 
p_{3}\sin \left( -\alpha \right) & -\lambda _{1}p_{2}\sin \left( -\alpha
\right) & \lambda _{1}p_{1}\sin \left( -\alpha \right) & \cos \left( -\alpha
\right)%
\end{array}%
\right]
\end{equation*}%
According to this, $P^{-n}$ is achieved as
\begin{equation*} \small
\left[ 
\begin{array}{cccc}
\cos \left( -n\alpha \right) & -\lambda _{1}\lambda _{2}p_{1}\sin \left(
-n\alpha \right) & -\lambda _{1}\lambda _{3}p_{2}\sin \left( -n\alpha \right)
& -\lambda _{2}\lambda _{3}p_{3}\sin \left( -n\alpha \right) \\ 
p_{1}\sin \left( -n\alpha \right) & \cos \left( -n\alpha \right) & -\lambda
_{3}p_{3}\sin \left( -n\alpha \right) & \lambda _{3}p_{2}\sin \left(
-n\alpha \right) \\ 
p_{2}\sin \left( -n\alpha \right) & \lambda _{2}p_{3}\sin \left( -n\alpha
\right) & \cos \left( -n\alpha \right) & -\lambda _{2}p_{1}\sin \left(
-n\alpha \right) \\ 
p_{3}\sin \left( -n\alpha \right) & -\lambda _{1}p_{2}\sin \left( -n\alpha
\right) & \lambda _{1}p_{1}\sin \left( -n\alpha \right) & \cos \left(
-n\alpha \right)%
\end{array}%
\right].
\end{equation*}%
\end{proof}

\begin{ex}
Let a 3PGUQ be 
\begin{equation*}
p=-\frac{1}{2}+\frac{1}{2}\left( \frac{1}{\sqrt{\lambda _{1}\lambda _{2}}},%
\frac{1}{\sqrt{\lambda _{1}\lambda _{3}}},\frac{1}{\sqrt{\lambda _{2}\lambda
_{3}}}\right).
\end{equation*}%
Polar representation of $p$ can be expressed as 
\begin{equation*}
p=\cos \frac{2\pi }{3}+\frac{1}{\sqrt{3}}\left( \frac{1}{\sqrt{\lambda
_{1}\lambda _{2}}},\frac{1}{\sqrt{\lambda _{1}\lambda _{3}}},\frac{1}{\sqrt{
\lambda _{2}\lambda _{3}}}\right) \sin \frac{2\pi }{3}.
\end{equation*}%
If 
\begin{equation*}
\hat{p}=\frac{1}{\sqrt{3}}\left( \frac{1}{\sqrt{\lambda _{1}\lambda _{2}}},%
\frac{1}{\sqrt{\lambda _{1}\lambda _{3}}},\frac{1}{\sqrt{\lambda _{2}\lambda
_{3}}}\right)
\end{equation*}%
then 
\begin{equation*}
p_{1}=\frac{1}{\sqrt{3\lambda _{1}\lambda _{2}}},\text{ \ \ }p_{2}=\frac{1}{ 
\sqrt{3\lambda _{1}\lambda _{3}}},\text{ \ \ }p_{3}=\frac{1}{\sqrt{3\lambda
_{2}\lambda _{3}}}.
\end{equation*}%
The matrix representation of $p$ is
\begin{equation*}
A=\left[ 
\begin{array}{cccc}
\frac{-1}{2} & \frac{-\sqrt{\lambda _{1}\lambda _{2}}}{2} & \frac{-\sqrt{
\lambda _{1}\lambda _{3}}}{2} & \frac{-\sqrt{\lambda _{2}\lambda _{3}}}{2}
\\ 
\frac{1}{2\sqrt{\lambda _{1}\lambda _{2}}} & -\frac{1}{2} & \frac{-\sqrt{
\lambda _{3}}}{2\sqrt{\lambda _{2}}} & \frac{\sqrt{\lambda _{3}}}{2\sqrt{
\lambda _{1}}} \\ 
\frac{1}{2\sqrt{\lambda _{1}\lambda _{3}}} & \frac{\sqrt{\lambda _{2}}}{2 
\sqrt{\lambda _{3}}} & -\frac{1}{2} & \frac{-\sqrt{\lambda _{2}}}{2\sqrt{
\lambda _{3}}} \\ 
\frac{1}{2\sqrt{\lambda _{2}\lambda _{3}}} & \frac{-\sqrt{\lambda _{1}}}{2 
\sqrt{\lambda _{3}}} & \frac{\sqrt{\lambda _{1}}}{2\sqrt{\lambda _{2}}} & -%
\frac{1}{2}%
\end{array}%
\right].
\end{equation*}%
and the matrix polar representation of $p$ is 
\begin{equation*}
A=\left[ 
\begin{array}{cccc}
\cos \frac{2\pi }{3} & -\lambda _{1}\lambda _{2}p_{1}\sin \frac{2\pi }{3} & 
-\lambda _{1}\lambda _{3}p_{2}\sin \frac{2\pi }{3} & -\lambda _{2}\lambda
_{3}p_{3}\sin \frac{2\pi }{3} \\ 
p_{1}\sin \frac{2\pi }{3} & \cos \frac{2\pi }{3} & -\lambda _{3}p_{3}\sin 
\frac{2\pi }{3} & \lambda _{3}p_{2}\sin \frac{2\pi }{3} \\ 
p_{2}\sin \frac{2\pi }{3} & \lambda _{2}p_{3}\sin \frac{2\pi }{3} & \cos 
\frac{2\pi }{3} & -\lambda _{2}p_{1}\sin \frac{2\pi }{3} \\ 
p_{3}\sin \frac{2\pi }{3} & -\lambda _{1}p_{2}\sin \frac{2\pi }{3} & \lambda
_{1}p_{1}\sin \frac{2\pi }{3} & \cos \frac{2\pi }{3}%
\end{array}%
\right].
\end{equation*}%
$5$th and $21$th powers of $p$ are 
\begin{eqnarray*}
p^{5} &=&\cos \left( 5.\frac{2\pi }{3}\right) +\frac{1}{\sqrt{3}}\left( 
\frac{1}{\sqrt{\lambda _{1}\lambda _{2}}},\frac{1}{\sqrt{\lambda _{1}\lambda
_{3}}},\frac{1}{\sqrt{\lambda _{2}\lambda _{3}}}\right) \sin \left( 5.\frac{
2\pi }{3}\right) \\
&=&\cos \left( \frac{\pi }{3}\right) +\frac{1}{\sqrt{3}}\left( \frac{1}{ 
\sqrt{\lambda _{1}\lambda _{2}}},\frac{1}{\sqrt{\lambda _{1}\lambda _{3}}}, 
\frac{1}{\sqrt{\lambda _{2}\lambda _{3}}}\right) \sin \left( \frac{\pi }{3}
\right) \\
&=&\frac{1}{2}+\frac{1}{2}\left( \frac{1}{\sqrt{\lambda _{1}\lambda _{2}}}, 
\frac{1}{\sqrt{\lambda _{1}\lambda _{3}}},\frac{1}{\sqrt{\lambda _{2}\lambda
_{3}}}\right)
\end{eqnarray*}%
and 
\begin{eqnarray*}
p^{21} &=&\cos \left( 21.\frac{2\pi }{3}\right) +\frac{1}{\sqrt{3}}\left( 
\frac{1}{\sqrt{\lambda _{1}\lambda _{2}}},\frac{1}{\sqrt{\lambda _{1}\lambda
_{3}}},\frac{1}{\sqrt{\lambda _{2}\lambda _{3}}}\right) \sin \left( 21.\frac{
2\pi }{3}\right) \\
&=&\cos 0+\frac{1}{\sqrt{3}}\left( \frac{1}{\sqrt{\lambda _{1}\lambda _{2}}}%
, \frac{1}{\sqrt{\lambda _{1}\lambda _{3}}},\frac{1}{\sqrt{\lambda
_{2}\lambda _{3}}}\right) \sin 0 \\
&=&1.
\end{eqnarray*}%
The matrix form is easily calculated as 
\begin{equation*}
A^{5}=\left[ 
\begin{array}{cccc}
\frac{1}{2} & \frac{-\sqrt{\lambda _{1}\lambda _{2}}}{2} & \frac{-\sqrt{
\lambda _{1}\lambda _{3}}}{2} & \frac{-\sqrt{\lambda _{2}\lambda _{3}}}{2}
\\ 
\frac{1}{2\sqrt{\lambda _{1}\lambda _{2}}} & \frac{1}{2} & \frac{-\sqrt{
\lambda _{3}}}{2\sqrt{\lambda _{2}}} & \frac{\sqrt{\lambda _{3}}}{2\sqrt{
\lambda _{1}}} \\ 
\frac{1}{2\sqrt{\lambda _{1}\lambda _{3}}} & \frac{\sqrt{\lambda _{2}}}{2 
\sqrt{\lambda _{3}}} & \frac{1}{2} & \frac{-\sqrt{\lambda _{2}}}{2\sqrt{
\lambda _{3}}} \\ 
\frac{1}{2\sqrt{\lambda _{2}\lambda _{3}}} & \frac{-\sqrt{\lambda _{1}}}{2 
\sqrt{\lambda _{3}}} & \frac{\sqrt{\lambda _{1}}}{2\sqrt{\lambda _{2}}} & 
\frac{1}{2}%
\end{array}%
\right]
\end{equation*}%
and 
\begin{equation*}
A^{21}=\left[ 
\begin{array}{cccc}
1 & 0 & 0 & 0 \\ 
0 & 1 & 0 & 0 \\ 
0 & 0 & 1 & 0 \\ 
0 & 0 & 0 & 1%
\end{array}%
\right]
\end{equation*}
\end{ex}

\subsection{Euler's formula for 3PGQs}

For any $v\in S_{\mathbb{K}}^{2}$, we know $v^{2}=-1$. Then 
\begin{equation*}
v^{3}=-v,\text{ \ \ }v^{4}=1,\text{ \ \ }v^{5}=v,\text{ \ \ }v^{6}=-1,\text{
\ }\ldots
\end{equation*}%
Euler's formula for 3PGQs with any angle $\theta $ is obtained as 
\begin{eqnarray*}
e^{v\theta } &=&1+v\theta +v^{2}\frac{\theta ^{2}}{2}+v^{3}\frac{\theta
^{3}}{3!}+v^{4}\frac{\theta ^{4}}{4!}+\cdots \\
&&1+v\theta -\frac{\theta ^{2}}{2}-v\frac{\theta ^{3}}{3!}+\frac{\theta
^{4}}{4!}+\cdots \\
&=&1-\frac{\theta ^{2}}{2!}+\frac{\theta ^{4}}{4!}-\cdots +v\left( \theta
-\frac{\theta ^{3}}{3!}+\frac{\theta ^{5}}{5!}-\cdots \right) \\
&=&\cos \theta +v\sin \theta.
\end{eqnarray*}

\subsection{Euler's formula for associated matrices with 3PGQs}

Let us choose any matrix $\mathcal{P}$ as follows: 
\begin{equation*}
\mathcal{P}=\left[ 
\begin{array}{cccc}
0 & -\lambda _{1}\lambda _{2}p_{1} & -\lambda _{1}\lambda _{3}p_{2} & 
-\lambda _{2}\lambda _{3}p_{3} \\ 
p_{1} & 0 & -\lambda _{3}p_{3} & \lambda _{3}p_{2} \\ 
p_{2} & \lambda _{2}p_{3} & 0 & -\lambda _{2}p_{1} \\ 
p_{3} & -\lambda _{1}p_{2} & \lambda _{1}p_{1} & 0%
\end{array}%
\right]
\end{equation*}%
Since $\lambda _{1}\lambda _{2}p_{1}^{2}+\lambda _{1}\lambda
_{3}p_{2}^{2}+\lambda _{2}\lambda _{3}p_{3}^{2}=1$, $\mathcal{P}^{2}=-I_{4}$
is easy to see. Then 
\begin{eqnarray*}
e^{\mathcal{P}\alpha } &=&I_{4}+\mathcal{P}\alpha +\frac{\left( \mathcal{P}
\alpha \right) ^{2}}{2!}+\frac{\left( \mathcal{P}\alpha \right) ^{3}}{3!}+ 
\frac{\left( \mathcal{P}\alpha \right) ^{4}}{4!} \\
&=&I_{4}\left( 1-\frac{\alpha ^{2}}{2!}+\frac{\alpha
^{4}}{4!}-\cdots \right) +\mathcal{P}\left( \alpha -\frac{\alpha ^{3}}{2!}+ 
\frac{\alpha ^{5}}{3!}-\cdots \right) \\
&=&\cos \alpha +\mathcal{P}\sin \alpha \\
&=&\cos \alpha +\left[ 
\begin{array}{cccc}
0 & -\lambda _{1}\lambda _{2}p_{1} & -\lambda _{1}\lambda _{3}p_{2} & 
-\lambda _{2}\lambda _{3}p_{3} \\ 
p_{1} & 0 & -\lambda _{3}p_{3} & \lambda _{3}p_{2} \\ 
p_{2} & \lambda _{2}p_{3} & 0 & -\lambda _{2}p_{1} \\ 
p_{3} & -\lambda _{1}p_{2} & \lambda _{1}p_{1} & 0%
\end{array}
\right] \sin \alpha \\
 &=&\left[ 
\begin{array}{cccc}
\cos \alpha & -\lambda _{1}\lambda _{2}p_{1}\sin \alpha & -\lambda
_{1}\lambda _{3}p_{2}\sin \alpha & -\lambda _{2}\lambda _{3}p_{3}\sin \alpha
\\ 
p_{1}\sin \alpha & \cos \alpha & -\lambda _{3}p_{3}\sin \alpha & \lambda
_{3}p_{2}\sin \alpha \\ 
p_{2}\sin \alpha & \lambda _{2}p_{3}\sin \alpha & \cos \alpha & -\lambda
_{2}p_{1}\sin \alpha \\ 
p_{3}\sin \alpha & -\lambda _{1}p_{2}\sin \alpha & \lambda _{1}p_{1}\sin
\alpha & \cos \alpha%
\end{array}
\right] \\
&=&P.
\end{eqnarray*}

\subsection{\textit The {n}th roots of matrices associated with 3PGQs}

\begin{equation*}
{\tiny A=\left[ 
\begin{array}{cccc}
\cos \left( \alpha \text{+}2k\pi \right) & -\lambda _{1}\lambda
_{2}p_{1}\sin \left( \alpha \text{+}2k\pi \right) & -\lambda _{1}\lambda
_{3}p_{2}\sin \left( \alpha \text{+}2k\pi \right) & -\lambda _{2}\lambda
_{3}p_{3}\sin \left( \alpha \text{+}2k\pi \right) \\ 
p_{1}\sin \left( \alpha \text{+}2k\pi \right) & \cos \left( \alpha \text{+}%
2k\pi \right) & -\lambda _{3}p_{3}\sin \left( \alpha \text{+}2k\pi \right) & 
\lambda _{3}p_{2}\sin \left( \alpha \text{+}2k\pi \right) \\ 
p_{2}\sin \left( \alpha \text{+}2k\pi \right) & \lambda _{2}p_{3}\sin \left(
\alpha \text{+}2k\pi \right) & \cos \left( \alpha \text{+}2k\pi \right) & 
-\lambda _{2}p_{1}\sin \left( \alpha \text{+}2k\pi \right) \\ 
p_{3}\sin \left( \alpha \text{+}2k\pi \right) & -\lambda _{1}p_{2}\sin
\left( \alpha \text{+}2k\pi \right) & \lambda _{1}p_{1}\sin \left( \alpha 
\text{+}2k\pi \right) & \cos \left( \alpha \text{+}2k\pi \right)%
\end{array}%
\right]}
\end{equation*}%
where $k\in \mathbb{Z}.$ The equation $X^{n}=A$ has $n$ roots. These roots
are found as $A_{k}^{\frac{1}{n}}=$ 
\begin{equation*}
{\tiny \left[ 
\begin{array}{cccc}
\cos \left( \frac{\alpha +2k\pi }{n}\right) & -\lambda _{1}\lambda
_{2}p_{1}\sin \left( \frac{\alpha +2k\pi }{n}\right) & -\lambda _{1}\lambda
_{3}p_{2}\sin \left( \frac{\alpha +2k\pi }{n}\right) & -\lambda _{2}\lambda
_{3}p_{3}\sin \left( \frac{\alpha +2k\pi }{n}\right) \\ 
p_{1}\sin \left( \frac{\alpha +2k\pi }{n}\right) & \cos \left( \frac{\alpha
+2k\pi }{n}\right) & -\lambda _{3}p_{3}\sin \left( \frac{\alpha +2k\pi }{n}%
\right) & \lambda _{3}p_{2}\sin \left( \frac{\alpha +2k\pi }{n}\right) \\ 
p_{2}\sin \left( \frac{\alpha +2k\pi }{n}\right) & \lambda _{2}p_{3}\sin
\left( \frac{\alpha +2k\pi }{n}\right) & \cos \left( \frac{\alpha +2k\pi }{n}%
\right) & -\lambda _{2}p_{1}\sin \left( \frac{\alpha +2k\pi }{n}\right) \\ 
p_{3}\sin \left( \frac{\alpha +2k\pi }{n}\right) & -\lambda _{1}p_{2}\sin
\left( \frac{\alpha +2k\pi }{n}\right) & \lambda _{1}p_{1}\sin \left( \frac{
\alpha +2k\pi }{n}\right) & \cos \left( \frac{\alpha +2k\pi }{n}\right)%
\end{array}%
\right]}
\end{equation*}%
For $k=0$, the first root is 
\begin{equation*}
A_{0}^{\frac{1}{n}}=\left[ 
\begin{array}{cccc}
\cos \left( \frac{\alpha }{n}\right) & -\lambda _{1}\lambda _{2}p_{1}\sin
\left( \frac{\alpha }{n}\right) & -\lambda _{1}\lambda _{3}p_{2}\sin \left( 
\frac{\alpha }{n}\right) & -\lambda _{2}\lambda _{3}p_{3}\sin \left( \frac{
\alpha }{n}\right) \\ 
p_{1}\sin \left( \frac{\alpha }{n}\right) & \cos \left( \frac{\alpha }{n}%
\right) & -\lambda _{3}p_{3}\sin \left( \frac{\alpha }{n}\right) & \lambda
_{3}p_{2}\sin \left( \frac{\alpha }{n}\right) \\ 
p_{2}\sin \left( \frac{\alpha }{n}\right) & \lambda _{2}p_{3}\sin \left( 
\frac{\alpha }{n}\right) & \cos \left( \frac{\alpha }{n}\right) & -\lambda
_{2}p_{1}\sin \left( \frac{\alpha }{n}\right) \\ 
p_{3}\sin \left( \frac{\alpha }{n}\right) & -\lambda _{1}p_{2}\sin \left( 
\frac{\alpha }{n}\right) & \lambda _{1}p_{1}\sin \left( \frac{\alpha }{n}%
\right) & \cos \left( \frac{\alpha }{n}\right)%
\end{array}%
\right],
\end{equation*}%
For $k=1$, the second root is 
\begin{equation*}
{\tiny A_{1}^{\frac{1}{n}}=\left[ 
\begin{array}{cccc}
\cos \left( \frac{\alpha +2\pi }{n}\right) & -\lambda _{1}\lambda
_{2}p_{1}\sin \left( \frac{\alpha +2\pi }{n}\right) & -\lambda _{1}\lambda
_{3}p_{2}\sin \left( \frac{\alpha +2\pi }{n}\right) & -\lambda _{2}\lambda
_{3}p_{3}\sin \left( \frac{\alpha +2\pi }{n}\right) \\ 
p_{1}\sin \left( \frac{\alpha +2\pi }{n}\right) & \cos \left( \frac{\alpha
+2\pi }{n}\right) & -\lambda _{3}p_{3}\sin \left( \frac{\alpha +2\pi }{n}%
\right) & \lambda _{3}p_{2}\sin \left( \frac{\alpha +2\pi }{n}\right) \\ 
p_{2}\sin \left( \frac{\alpha +2\pi }{n}\right) & \lambda _{2}p_{3}\sin
\left( \frac{\alpha +2\pi }{n}\right) & \cos \left( \frac{\alpha +2\pi }{n}%
\right) & -\lambda _{2}p_{1}\sin \left( \frac{\alpha +2\pi }{n}\right) \\ 
p_{3}\sin \left( \frac{\alpha +2\pi }{n}\right) & -\lambda _{1}p_{2}\sin
\left( \frac{\alpha +2\pi }{n}\right) & \lambda _{1}p_{1}\sin \left( \frac{
\alpha +2\pi }{n}\right) & \cos \left( \frac{\alpha +2\pi }{n}\right)%
\end{array}%
\right]}.
\end{equation*}%
Similarly, for $k=n-1$, the $n$th root is obtained.

\subsection{Relations among the powers of matrices associated with 3PGQs}

\begin{thm}
Let $m=\frac{2\pi }{\theta }\in \mathbb{Z}^{+}-\left\{ 1\right\} $ and
polar expression of any 3PGUQ $p$ be $p=\cos \theta +v\sin \theta$. Then $n\equiv s\left( {mod}\ m\right) $ if and only if $p^{n}=p^{s}$.
\end{thm}

\begin{proof}
Let $n\equiv s\left( {mod}\ m\right) $.
\begin{eqnarray*}
	p^{n} &=&\left( \cos \theta +\hat{p}\sin \theta \right) ^{n} \\
	&=&\cos \left( n\theta \right) +\hat{p}\sin \left( n\theta \right) \\
	&=&\cos \left( \left( mk+s\right) \theta \right) +\hat{p}\sin \left( \left(
	mk+s\right) \theta \right) \\
	&=&\cos \left( \left( \frac{2\pi }{\theta }k+s\right) \theta \right) +\hat{%
		p}\sin \left( \left( \frac{2\pi }{\theta }k+s\right) \theta \right) \\
	&=&\cos \left( 2\pi k+s\theta \right) +\hat{p}\sin \left( 2\pi k+s\theta
	\right) \\
	&=&\cos \left( s\theta \right) +\hat{p}\sin \left( s\theta \right) \\
	&=&\left( \cos \theta +\hat{p}\sin \theta \right) ^{s} \\
	&=&p^{s}.
\end{eqnarray*}%
On the other hand, let $p^{n}=\cos \left( n\theta \right) +\hat{p}\sin \left(
n\theta \right) $ and $p^{s}=\cos \left( s\theta \right) +\hat{p}\sin
\left( s\theta \right) $. Since $p^{n}=p^{s}$, $\cos \left(
n\theta \right) =\cos \left( s\theta \right) $ and $\sin \left( n\theta
\right) =\sin \left( s\theta \right) $ are found. This also requires the equation
\begin{equation*}
n\theta =s\theta +2k\pi ,\text{ }k\in \mathbb{Z}.
\end{equation*}%
Thus
\begin{equation*}
n=\frac{2\pi }{\theta }k+s,\text{ \ }n\equiv s\left( {mod}\ m\right)
\end{equation*}%
is attained.
\end{proof}

\begin{ex}
Let $p=-\frac{1}{2}+\frac{1}{2}\left( \frac{1}{\sqrt{\lambda _{1}\lambda _{2}}},%
\frac{1}{\sqrt{\lambda _{1}\lambda _{3}}},\frac{1}{\sqrt{\lambda _{2}\lambda
		_{3}}}\right)$ be a 3PGUQ. We have expressed the polar form of $p$ in the previous example. Since $%
\varphi =\frac{2\pi }{3}$ from Theorem 4.6 we find $m=\frac{2\pi }{2\pi /3%
}=3$. Then we have 
\begin{eqnarray*}
p &=&p^{4}=p^{7}=\cdots \\
p^{2} &=&p^{5}=p^{8}=\cdots \\
p^{3} &=&p^{6}=p^{9}=\cdots =1.
\end{eqnarray*}
\end{ex}

\begin{thm}
Let the expression of 3PGUQ $p$ in the polar form be \linebreak $p=\cos
\theta +v\sin \theta$, $m=\frac{2\pi }{\theta }\in \mathbb{Z}^{+}-\left\{
1\right\} $ and let $A$ be the matrix representation of $p$. Accordingly, $%
n\equiv s\left( {mod}\ m\right) $ if and only if $A^{n}=A^{s}$.
\end{thm}

\begin{proof}
Similar to the proof of Theorem 4.6, it can easily be proved.
\end{proof}

\begin{ex}
The matrix representation of a 3PGUQ 
\begin{equation*}
p=\frac{1}{\sqrt{2}}+\frac{1}{2}\left( \frac{1}{\sqrt{\lambda _{1}\lambda
_{2}}},\frac{-1}{\sqrt{2\lambda _{1}\lambda _{3}}},\frac{1}{\sqrt{2\lambda
_{2}\lambda _{3}}}\right)
\end{equation*}%
is found as 
\begin{equation*}
A=\left[ 
\begin{array}{cccc}
\frac{1}{\sqrt{2}} & \frac{-\sqrt{\lambda _{1}\lambda _{2}}}{2} & \frac{ 
\sqrt{\lambda _{1}\lambda _{3}}}{2\sqrt{2}} & \frac{-\sqrt{\lambda
_{2}\lambda _{3}}}{2\sqrt{2}} \\ 
\frac{1}{2\sqrt{\lambda _{1}\lambda _{2}}} & \frac{1}{\sqrt{2}} & \frac{- 
\sqrt{\lambda _{3}}}{2\sqrt{2\lambda _{2}}} & \frac{-\sqrt{\lambda _{3}}}{2 
\sqrt{2\lambda _{1}}} \\ 
\frac{-1}{2\sqrt{2\lambda _{1}\lambda _{3}}} & \frac{\sqrt{\lambda _{2}}}{2 
\sqrt{2\lambda _{3}}} & \frac{1}{\sqrt{2}} & \frac{-\sqrt{\lambda _{2}}}{2 
\sqrt{\lambda _{1}}} \\ 
\frac{1}{2\sqrt{2\lambda _{2}\lambda _{3}}} & \frac{\sqrt{\lambda _{1}}}{2 
\sqrt{2\lambda _{3}}} & \frac{\sqrt{\lambda _{1}}}{2\sqrt{\lambda _{2}}} & 
\frac{1}{\sqrt{2}}%
\end{array}%
\right].
\end{equation*}
From Theorem 5.7, $m=\frac{2\pi }{\pi /4}=8$ which means that it is as
follows. 
\begin{eqnarray*}
A &=&A^{9}=A^{17}=\cdots \\
A^{2} &=&A^{10}=A^{18}=\cdots \\
&&\cdots \\
A^{8} &=&A^{16}=A^{24}=\cdots =I_{4}.
\end{eqnarray*}
It is also possible to obtain the square roots of the matrix $A$: 
\begin{equation*}
{\tiny \left[ 
\begin{array}{cccc}
\cos \frac{\pi /4+2k\pi }{2} & -\lambda _{1}\lambda _{2}p_{1}\sin \frac{\pi
/4+2k\pi }{2} & -\lambda _{1}\lambda _{3}p_{2}\sin \frac{\pi /4+2k\pi }{2} & 
-\lambda _{2}\lambda _{3}p_{3}\sin \frac{\pi /4+2k\pi }{2} \\ 
p_{1}\sin \frac{\pi /4+2k\pi }{2} & \cos \frac{\pi /4+2k\pi }{2} & -\lambda
_{3}p_{3}\sin \frac{\pi /4+2k\pi }{2} & \lambda _{3}p_{2}\sin \frac{\pi
/4+2k\pi }{2} \\ 
p_{2}\sin \frac{\pi /4+2k\pi }{2} & \lambda _{2}p_{3}\sin \frac{\pi /4+2k\pi 
}{2} & \cos \frac{\pi /4+2k\pi }{2} & -\lambda _{2}p_{1}\sin \frac{\pi
/4+2k\pi }{2} \\ 
p_{3}\sin \frac{\pi /4+2k\pi }{2} & -\lambda _{1}p_{2}\sin \frac{\pi
/4+2k\pi }{2} & \lambda _{1}p_{1}\sin \frac{\pi /4+2k\pi }{2} & \cos \frac{
\pi /4+2k\pi }{2}%
\end{array}%
\right].}
\end{equation*}%
For $k=0$, the first square root is 
\begin{equation*}
A_{0}^{\frac{1}{2}}=\left[ 
\begin{array}{cccc}
\cos \frac{\pi }{8} & -\lambda _{1}\lambda _{2}p_{1}\sin \frac{\pi }{8} & 
-\lambda _{1}\lambda _{3}p_{2}\sin \frac{\pi }{8} & -\lambda _{2}\lambda
_{3}p_{3}\sin \frac{\pi }{8} \\ 
p_{1}\sin \frac{\pi }{8} & \cos \frac{\pi }{8} & -\lambda _{3}p_{3}\sin 
\frac{\pi }{8} & \lambda _{3}p_{2}\sin \frac{\pi }{8} \\ 
p_{2}\sin \frac{\pi }{8} & \lambda _{2}p_{3}\sin \frac{\pi }{8} & \cos \frac{
\pi }{8} & -\lambda _{2}p_{1}\sin \frac{\pi }{8} \\ 
p_{3}\sin \frac{\pi }{8} & -\lambda _{1}p_{2}\sin \frac{\pi }{8} & \lambda
_{1}p_{1}\sin \frac{\pi }{8} & \cos \frac{\pi }{8}%
\end{array}%
\right].
\end{equation*}
For $k=1$, the other square root is 
\begin{equation*}
A_{1}^{\frac{1}{2}}=\left[ 
\begin{array}{cccc}
\cos \frac{9\pi }{8} & -\lambda _{1}\lambda _{2}p_{1}\sin \frac{9\pi }{8} & 
-\lambda _{1}\lambda _{3}p_{2}\sin \frac{9\pi }{8} & -\lambda _{2}\lambda
_{3}p_{3}\sin \frac{9\pi }{8} \\ 
p_{1}\sin \frac{9\pi }{8} & \cos \frac{9\pi }{8} & -\lambda _{3}p_{3}\sin 
\frac{9\pi }{8} & \lambda _{3}p_{2}\sin \frac{9\pi }{8} \\ 
p_{2}\sin \frac{9\pi }{8} & \lambda _{2}p_{3}\sin \frac{9\pi }{8} & \cos 
\frac{9\pi }{8} & -\lambda _{2}p_{1}\sin \frac{9\pi }{8} \\ 
p_{3}\sin \frac{9\pi }{8} & -\lambda _{1}p_{2}\sin \frac{9\pi }{8} & \lambda
_{1}p_{1}\sin \frac{9\pi }{8} & \cos \frac{9\pi }{8}%
\end{array}%
\right].
\end{equation*}
Besides, it is easy to see that $A_{0}^{\frac{1}{2}}+A_{1}^{\frac{1}{2}}=0$.
\end{ex}
\begin{thm}
Let the polar representation of a 3PGUQ $p$ be \linebreak $p=\sqrt{N_{p}}%
\left( \cos \theta +\hat{p}\sin \theta \right) $ and $\frac{2\pi }{\theta }%
=m\in \mathbb{Z}^{+}-\left\{ 1\right\} $. $n\equiv s\left( {mod} \ m\right) $
if and only if $p^{n}=\left( \sqrt{N_{p}}\right) ^{n-s}p^{s}$.
\end{thm}

\begin{proof}
Similar to the proof of the theorem in \cite{mer} this can be proved.
\end{proof}

\section{Lie Algebra And Matrix Representations of 3PGQs}

In \cite{kar}, Karger and Novak show that the set of all unit quaternions is
a 3-dimensional Lie group and also study to Lie algebra. Jafari and Yayl\i \
conducted the same study on 2-parameter generalized unit
quaternions  in \cite{jaf1}. In this section we will show
that the set of 3PGUQs is a Lie group and give the
properties of the Lie algebra, adjoint mappings, Lie multpilication and Killing bi-linear form for 3PGQs.

\subsection{Lie group of 3PGQs}
\begin{thm}
	$S_{\mathbb{K}}=\left\{ p\in \mathbb{K}:N_{p}=1\right\} $ is a 3-dimensional
	Lie group.
\end{thm}

\begin{proof}
	
	The set	$S_{\mathbb{K}}$ is a unity group , with together multiplication on 3PGQ. The unit element of group $S_{%
		\mathbb{K}}$ is $e=1$. Let us define the function $f$ as 
	\begin{equation*}
	\begin{array}{cc}
	f: & \mathbb{K}\rightarrow \mathbb{R\hspace*{6.0cm}} \\ 
	& p\rightarrow f\left( p\right) =a_{0}+\lambda _{1}\lambda
	_{2}a_{1}^{2}+\lambda _{1}\lambda _{3}a_{2}^{2}+\lambda _{2}\lambda
	_{3}a_{3}^{2}.%
	\end{array}%
	\end{equation*}%
	The function $f$ is expressed as coordinate functions as the following:
	\begin{equation*}
	f=x_{0}^{2}+\lambda _{1}\lambda _{2}x_{1}^{2}+\lambda _{1}\lambda
	_{3}x_{2}^{2}+\lambda _{2}\lambda _{3}x_{3}^{2}
	\end{equation*}%
	The Jacobi matrix of the function $f$ can be written as
	\begin{equation*}
	J\left( f\right) =\left[ 
	\begin{array}{cccc}
	2x_{0} & 2\lambda _{1}\lambda _{2}x_{1} & 2\lambda _{1}\lambda _{3}x_{2} & 
	2\lambda _{2}\lambda _{3}x_{3}%
	\end{array}%
	\right]
	\end{equation*}%
	$RankJ\left( f\right) =1$.  $f^{-1}\left( 1\right) $ is a submanifold of $\mathbb{K}$. 	It can be shown that the mappings are defined as follows are  differentiable: 
	\begin{equation*}
	\begin{array}{cc}
	\gamma : & \mathbb{K}\times \mathbb{K}\rightarrow 
	\mathbb{K}\mathbb{\hspace*{1.8cm}} \\ 
	& \left( p,q\right) \rightarrow \gamma \left( p,q\right) =pq%
	\end{array}%
	\text{ \ and \ \ }%
	\begin{array}{cc}
	\eta : & \mathbb{K}\rightarrow \mathbb{K}\mathbb{%
		\hspace*{2.2cm}} \\ 
	& p\rightarrow \eta \left( p\right) =p^{-1}=\bar{p}%
	\end{array}.
	\end{equation*}%
	$S_{\mathbb{K}}$ is a 3-dimensional Lie group.
\end{proof}

\begin{thm}
	\textit{\ }$Im\left( \mathbb{K}\right) $ is Lie algebra of Lie group $S_{%
		\mathbb{K}}$.
\end{thm}

\begin{proof}
	
	Let	$T_{S_{\mathbb{K}}}\left( e\right) $ be the set of velocity vectors through point $e$ and \linebreak $v_{e}\in T_{S_{\mathbb{K}%
	}}\left( e\right) $. Let us define a curve $\rho$ on $S_{\mathbb{K}}$ as:
	\begin{equation*}
	\begin{array}{cc}
	\rho : & I\subset \mathbb{R}\rightarrow S_{\mathbb{K}}\mathbb{\hspace*{6.75cm}} \\ 
	& \rho \left( s\right) \rightarrow \rho \left( s\right) =a_{0}\left(
	s\right) +a_{1}\left( s\right) e_{1}+a_{2}\left( s\right) e_{2}+a_{3}\left(
	s\right) e_{3}.%
	\end{array}%
	\end{equation*}%
	Let the curve $\rho$ accept $v_{e}$ as a velocity vector.%
	\begin{equation*}
	\rho \left( 0\right) =1\text{ and }\rho ^{^{\prime }}\left( 0\right) =v_{e}
	\end{equation*}%
	Since	$\rho \left( s\right) \in S_{\mathbb{K}}$
	\begin{equation}
	a_{0}^{2}\left( s\right) +\lambda _{1}\lambda _{2}a_{1}^{2}\left( s\right)
	+\lambda _{1}\lambda _{3}a_{2}^{2}\left( s\right) +\lambda _{2}\lambda
	_{3}a_{3}^{2}\left( s\right) =1.  \label{7}
	\end{equation}%
	At point $s=0$, derivative of Eq.(\ref{7}) is 
	\begin{equation*}
	2a_{0}^{^{\prime }}\left( s\right) a_{0}\left( s\right) +2\lambda
	_{1}\lambda _{2}a_{1}^{^{\prime }}\left( s\right) a_{1}\left( s\right)
	+2\lambda _{1}\lambda _{3}a_{2}^{^{\prime }}\left( s\right) a_{2}\left(
	s\right) +2\lambda _{2}\lambda _{3}a_{3}^{^{\prime }}\left( s\right)
	a_{3}\left( s\right) =0.
	\end{equation*}%
	Then
	\begin{equation*}
	a_{0}\left( 0\right) =1,a_{1}\left( 0\right) =0,a_{2}\left( 0\right)
	=0,a_{3}\left( 0\right) =0,a_{0}^{^{\prime }}\left( 0\right) =0.
	\end{equation*}
	All of the vectors on $T_{S_{\mathbb{K}}}\left( e\right) $ can be written as a linear combination of the vectors in the base
	\begin{equation*}
	\left. \left\{ \dfrac{\partial }{\partial x_{1}},\dfrac{\partial }{\partial
		x_{2}},\dfrac{\partial }{\partial x_{3}}\right\} \right\vert _{s=0}
	\end{equation*}%
	of the tangent space at point $e$ of $%
	\text{Im}\left( \mathbb{K}\right) $. Then velocity vector 
	$a_{0}^{^{\prime }}\left( 0\right) =0$ is written as 
	\begin{equation*}
	a_{0}^{^{\prime }}=a_{0}^{^{\prime }}(0)\dfrac{\partial }{\partial x_{0}}%
	+a_{1}^{^{\prime }}(0)\dfrac{\partial }{\partial x_{1}}+a_{2}^{^{\prime }}(0)%
	\dfrac{\partial }{\partial x_{2}}+a_{3}^{^{\prime }}(0)\dfrac{\partial }{%
		\partial x_{3}}
	\end{equation*}%
	Since $a_{0}^{^{\prime }}\left( 0\right) =0$,
	\begin{equation*}
	T_{S_{\mathbb{K}}}\left( e\right) \subset Sp\left\{ \dfrac{\partial }{%
		\partial x_{1}},\dfrac{\partial }{\partial x_{2}},\dfrac{\partial }{\partial
		x_{3}}\right\}
	\end{equation*}%
	is found. Also since  $boyS_{\mathbb{K}}=boyT_{S_{\mathbb{K}}}\left( e\right) =3$, we obtain 
	\begin{equation*}
	T_{S_{\mathbb{K}}}\left( e\right) =Sp\left\{ \dfrac{\partial }{\partial x_{1}%
	},\dfrac{\partial }{\partial x_{2}},\dfrac{\partial }{\partial x_{3}}\right\}.
	\end{equation*}%
	Therefore, Lie algebra of Lie group $S_{\mathbb{K}}$ is $\text{Im}\left( 
	\mathbb{K}\right) $.
\end{proof}

\begin{cor}
	$T_{S_{\mathbb{K}}}\left( e\right)$ is isomorphic to \newline $\text{Im}\left( \mathbb{K}\right)
	=\left\{ a_{1}e_{1}+a_{2}e_{2}+a_{3}e_{3}\mid a_{1},a_{2},a_{3}\in \mathbb{R}%
	\right\}$.
\end{cor}

\subsection{Adjoint mappings for Lie algebra and Lie group on 3PGQs}

\subsubsection{Matrix representation for Lie group of $S_{\mathbb{K}}$}

For $S_{\mathbb{K}}=\left\{ q\in \mathbb{K}:N_{q}=1\right\} $ and $k\in S_{%
	\mathbb{K}}$, let us define function $g_{k}$ that is bijective and
differentiable as 
\begin{equation*}
\begin{array}{cc}
g_{k}: & S_{\mathbb{K}}\rightarrow S_{\mathbb{K}}\mathbb{\hspace*{2cm}} \\ 
& x\rightarrow g_{k}\left( x\right) =kxk^{-1}%
\end{array}.
\end{equation*}%
Let us consider derivative map of the function and restriction of about the unit 
$e=1$ point of the group, for any $p$ in $S_{\mathbb{K}}$, the following mapping
is called the adjoint mapping: 
\begin{equation*}
\begin{array}{cc}
\mathcal{A}dp: & T_{S_{\mathbb{K}}}\left( e\right) \rightarrow T_{S_{\mathbb{%
			K}}}\left( e\right) \hspace*{0.75cm} \\ 
& q\rightarrow pqp^{-1}%
\end{array}%
\end{equation*}%
Since $T_{S_{\mathbb{K}}}\left( e\right) =Sp\left\{
e_{1},e_{2},e_{3}\right\} $, according to the base $\left\{
e_{1},e_{2},e_{3}\right\} $ we have 
\begin{eqnarray*}
	\mathcal{A}dp\left( e_{1}\right) &=&\left( a_{0}^{2}+\lambda _{1}\lambda
	_{2}a_{1}^{2}-\lambda _{1}\lambda _{3}a_{2}^{2}-\lambda _{2}\lambda
	_{3}a_{3}^{2}\right) e_{1} \\
	&+&\left( 2\lambda _{1}\lambda _{2}a_{1}a_{2}+2\lambda _{2}a_{0}a_{3}\right)
	e_{2}+\left( 2\lambda _{1}\lambda _{2}a_{1}a_{2}-2\lambda
	_{1}a_{0}a_{2}\right) e_{3} \\
	\mathcal{A}dp\left( e_{2}\right) &=&\left( 2\lambda _{1}\lambda
	_{3}a_{1}a_{2}-2\lambda _{3}a_{0}a_{3}\right) e_{1} \\
	&+&\left( a_{0}^{2}-\lambda _{1}\lambda _{2}a_{1}^{2}+\lambda _{1}\lambda
	_{3}a_{2}^{2}-\lambda _{2}\lambda _{3}a_{3}^{2}\right) e_{2} \\
	&+&\left( 2\lambda _{1}\lambda _{3}a_{2}a_{3}-2\lambda _{1}a_{0}a_{1}\right)
	e_{3} \\
	\mathcal{A}dp\left( e_{3}\right) &=&\left( 2\lambda _{2}\lambda
	_{3}a_{1}a_{3}+2\lambda _{3}a_{0}a_{2}\right) e_{1}+\left( 2\lambda
	_{2}\lambda _{3}a_{2}a_{3}-2\lambda _{2}a_{0}a_{1}\right) e_{2} \\
	&&+\left( a_{0}^{2}-\lambda _{1}\lambda _{2}a_{1}^{2}-\lambda _{1}\lambda
	_{3}a_{2}^{2}+\lambda _{2}\lambda _{3}a_{3}^{2}\right) e_{3}.
\end{eqnarray*}%
Consequently, the matrix is obtained as 
\begin{eqnarray*}
	Adp =\left[ {\small 
		\begin{array}{cc}
			a_{0}^{2}+\lambda _{1}\lambda _{2}a_{1}^{2}-\lambda _{1}\lambda
			_{3}a_{2}^{2}-\lambda _{2}\lambda _{3}a_{3}^{2} & 2\lambda _{1}\lambda
			_{3}a_{1}a_{2}-2\lambda _{3}a_{0}a_{3} \\ 
			2\lambda _{1}\lambda _{2}a_{1}a_{2}+2\lambda _{2}a_{0}a_{3} & 
			a_{0}^{2}-\lambda _{1}\lambda _{2}a_{1}^{2}+\lambda _{1}\lambda
			_{3}a_{2}^{2}-\lambda _{2}\lambda _{3}a_{3}^{2} \\ 
			2\lambda _{1}\lambda _{2}a_{1}a_{3}-2\lambda _{1}a_{0}a_{2} & 2\lambda
			_{1}\lambda _{3}a_{2}a_{3}+2\lambda _{1}a_{0}a_{1}%
		\end{array}%
	}\right. \\
	\\
	\left. {\small 
		\begin{array}{c}
			2\lambda _{2}\lambda _{3}a_{1}a_{3}+2\lambda _{3}a_{0}a_{2} \\ 
			2\lambda _{2}\lambda _{3}a_{2}a_{3}-2\lambda _{2}a_{0}a_{1} \\ 
			a_{0}^{2}-\lambda _{1}\lambda _{2}a_{1}^{2}-\lambda _{1}\lambda
			_{3}a_{2}^{2}+\lambda _{2}\lambda _{3}a_{3}^{2}%
		\end{array}%
	}\right].
\end{eqnarray*}%
\begin{thm}
	If 
	\begin{equation*}
	\varepsilon =\left[ 
	\begin{array}{ccc}
	\lambda _{1}\lambda _{2} & 0 & 0 \\ 
	0 & \lambda _{1}\lambda _{3} & 0 \\ 
	0 & 0 & \lambda _{2}\lambda _{3}%
	\end{array}
	\right]
	\end{equation*}
	then $\mathcal{A}dp^{T}\varepsilon \mathcal{A}dp=\varepsilon $.
\end{thm}

\begin{proof} If the matrix $\mathcal{A}dp^{T}$ is multiplied by the matrix $\varepsilon $ then  \newline
	$\mathcal{A}dp^{T}\varepsilon=$
	\begin{eqnarray*}
		\left[ {\tiny 
			\begin{array}{cc}
				\lambda _{1}\lambda _{2}a_{0}^{2}+\lambda _{1}^{2}\lambda
				_{2}^{2}a_{1}^{2}-\lambda _{1}^{2}\lambda _{2}\lambda _{3}a_{2}^{2}-\lambda
				_{1}\lambda _{2}^{2}\lambda _{3}a_{3}^{2} & 2\lambda _{1}\lambda
				_{2}^{2}a_{0}a_{3}+2\lambda _{1}^{2}\lambda _{2}^{2}a_{1}a_{2} \\ 
				2\lambda _{1}^{2}\lambda _{3}^{2}a_{1}a_{2}-2\lambda _{1}\lambda
				_{3}^{2}a_{0}a_{3} & \lambda _{1}\lambda _{3}a_{0}^{2}-\lambda
				_{1}^{2}\lambda _{2}\lambda _{3}a_{1}^{2}+\lambda _{1}^{2}\lambda
				_{3}^{2}a_{2}^{2}-\lambda _{1}\lambda _{2}\lambda _{3}^{2}a_{3}^{2} \\ 
				2\lambda _{2}\lambda _{3}^{2}a_{0}a_{2}+2\lambda _{2}^{2}\lambda
				_{3}^{2}a_{1}a_{3} & 2\lambda _{2}^{2}\lambda _{3}^{2}a_{2}a_{3}-2\lambda
				_{2}^{2}\lambda _{3}a_{0}a_{1}%
			\end{array}%
		}\right. \\ \\
		\left. {\tiny 
			\begin{array}{c}
				2\lambda _{1}^{2}\lambda _{2}^{2}a_{1}a_{3}-2\lambda _{1}^{2}\lambda
				_{2}a_{0}a_{2} \\ 
				2\lambda _{1}^{2}\lambda _{3}a_{0}a_{1}+2\lambda _{1}^{2}\lambda
				_{3}^{2}a_{2}a_{3} \\ 
				\lambda _{2}\lambda _{3}a_{0}^{2}-\lambda _{1}\lambda _{2}^{2}\lambda
				_{3}a_{1}^{2}-\lambda _{1}\lambda _{2}\lambda _{3}^{2}a_{2}^{2}+\lambda
				_{2}^{2}\lambda _{3}^{2}a_{3}^{2}%
			\end{array}%
		}\right] 
	\end{eqnarray*}%
	is found. From here 
	\begin{eqnarray*}
		\mathcal{A}dp^{T}\varepsilon \mathcal{A}dp &=&\left( N_{p}\right) ^{2}\left[ 
		\begin{array}{ccc}
			\lambda _{1}\lambda _{2} & 0 & 0 \\ 
			0 & \lambda _{1}\lambda _{3} & 0 \\ 
			0 & 0 & \lambda _{2}\lambda _{3}%
		\end{array}%
		\right]  \\
		&=&\varepsilon 
	\end{eqnarray*}%
	is achieved. The matrix $\mathcal{A}dp$ is orthogonal. In addition to 
	\begin{eqnarray*}
		\det \mathcal{A}dp &=&\left( a_{0}^{2}+\lambda _{1}\lambda
		_{2}a_{1}^{2}+\lambda _{1}\lambda _{3}a_{2}^{2}+\lambda _{2}\lambda
		_{3}a_{3}^{2}\right) ^{3}\allowbreak  \\
		&=&\left( N_{p}\right) ^{3}=1
	\end{eqnarray*}%
	is obtaiend. For this reason, the linear mapping $\mathcal{A}dp$ is an
	isometry on \linebreak $T_{G}(e)=\text{Im}\left( \mathbb{K}\right) $
\end{proof}

\begin{thm}
	Let $p$ be a 3PGUQ. For $i\in \left\{ 1,2,3\right\} $, if $\lambda _{i}>0$
	then 
	\begin{equation*}
	\mathcal{A}dp=I+\sin \theta S+\left( 1-\cos \theta \right) S^{2}.
	\end{equation*}
\end{thm}

\begin{proof} For $i\in \left\{ 1,2,3\right\} $, $\lambda _{i}>0$ and any $%
	V_{p}\in Im\left( \mathbb{K}\right) ,$ $V_{p}\neq 0\ $, 
	\begin{equation*}
	f\left( V_{p},V_{p}\right) =\lambda _{1}\lambda _{2}a_{1}^{2}+\lambda
	_{1}\lambda _{3}a_{2}^{2}+\lambda _{2}\lambda _{3}a_{3}^{2}>0
	\end{equation*}
	is found. Therefore the function $f$ is positive definite. If 
	\begin{equation*}
	p=a_{0}+a_{1}e_{1}+a_{2}e_{2}+a_{3}e_{3}\text{ and }N_{p}=1
	\end{equation*}
	then 
	\begin{eqnarray*}
		p &=&a_{0}+\sqrt{\lambda _{1}\lambda _{2}a_{1}^{2}+\lambda _{1}\lambda
			_{3}a_{2}^{2}+\lambda _{2}\lambda _{3}a_{3}^{2}}\frac{
			a_{1}e_{1}+a_{2}e_{2}+a_{3}e_{3}}{\sqrt{\lambda _{1}\lambda
				_{2}a_{1}^{2}+\lambda _{1}\lambda _{3}a_{2}^{2}+\lambda _{2}\lambda
				_{3}a_{3}^{2}}} \\
		p &=&\cos \frac{\theta }{2}+\hat{p}\sin \frac{\theta }{2}
	\end{eqnarray*}
	where 
	\begin{eqnarray*}
		\cos \frac{\theta }{2} &=&a_{0},\text{ }\sin \frac{\theta }{2}=\sqrt{\lambda
			_{1}\lambda _{2}a_{1}^{2}+\lambda _{1}\lambda _{3}a_{2}^{2}+\lambda
			_{2}\lambda _{3}a_{3}^{2}}, \\
		\hat{p} &=&\frac{a_{1}e_{1}+a_{2}e_{2}+a_{3}e_{3}}{\sqrt{\lambda
				_{1}\lambda _{2}a_{1}^{2}+\lambda _{1}\lambda _{3}a_{2}^{2}+\lambda
				_{2}\lambda _{3}a_{3}^{2}}} \in S_{\mathbb{K}}^{2}.
	\end{eqnarray*}
	Firstly, we need to find skew-symmetric matrix of the vector $\hat{p}$. If
	
	\begin{equation*}
	\varepsilon =\left[ 
	\begin{array}{ccc}
	\lambda _{1}\lambda _{2} & 0 & 0 \\ 
	0 & \lambda _{1}\lambda _{3} & 0 \\ 
	0 & 0 & \lambda _{2}\lambda _{3}%
	\end{array}%
	\right] \text{ }
	\end{equation*}%
	then a matrix $S$ that providing the proposition $\varepsilon S=S^{T}\left(
	-\varepsilon \right) $ must exist; 
	\begin{eqnarray*}
		\varepsilon S &=&\left[ 
		\begin{array}{ccc}
			\lambda _{1}\lambda _{2} & 0 & 0 \\ 
			0 & \lambda _{1}\lambda _{3} & 0 \\ 
			0 & 0 & \lambda _{2}\lambda _{3}%
		\end{array}%
		\right] \left[ 
		\begin{array}{ccc}
			0 & -\lambda _{3}s_{3} & \lambda _{3}s_{2} \\ 
			\lambda _{2}s_{3} & 0 & -\lambda _{2}s_{1} \\ 
			-\lambda _{1}s_{2} & \lambda _{1}s_{1} & 0%
		\end{array}%
		\right]  \\
		&=&\left[ 
		\begin{array}{ccc}
			0 & -\lambda _{1}\lambda _{2}\lambda _{3}s_{3} & \lambda _{1}\lambda
			_{2}\lambda _{3}s_{2} \\ 
			\lambda _{1}\lambda _{2}\lambda _{3}s_{3} & 0 & -\lambda _{1}\lambda
			_{2}\lambda _{3}s_{1} \\ 
			-\lambda _{1}\lambda _{2}\lambda _{3}s_{2} & \lambda _{1}\lambda _{2}\lambda
			_{3}s_{1} & 0%
		\end{array}%
		\right] \allowbreak \allowbreak , \\
		S^{T}(-\varepsilon ) &=&\left[ 
		\begin{array}{ccc}
			0 & \lambda _{2}s_{3} & -\lambda _{1}s_{2} \\ 
			-\lambda _{3}s_{3} & 0 & \lambda _{1}s_{1} \\ 
			\lambda _{3}s_{2} & -\lambda _{2}s_{1} & 0%
		\end{array}%
		\right] \left[ 
		\begin{array}{ccc}
			-\lambda _{1}\lambda _{2} & 0 & 0 \\ 
			0 & -\lambda _{1}\lambda _{3} & 0 \\ 
			0 & 0 & -\lambda _{2}\lambda _{3}%
		\end{array}%
		\right]  \\
		&=&\allowbreak \left[ 
		\begin{array}{ccc}
			0 & -\lambda _{1}\lambda _{2}\lambda _{3}s_{3} & \lambda _{1}\lambda
			_{2}\lambda _{3}s_{2} \\ 
			\lambda _{1}\lambda _{2}\lambda _{3}s_{3} & 0 & -\lambda _{1}\lambda
			_{2}\lambda _{3}s_{1} \\ 
			-\lambda _{1}\lambda _{2}\lambda _{3}s_{2} & \lambda _{1}\lambda _{2}\lambda
			_{3}s_{1} & 0%
		\end{array}%
		\right] \allowbreak \allowbreak .
	\end{eqnarray*}%
	The matrix $S$ is attained as follows:
	\begin{equation*}
	S=\left[ 
	\begin{array}{ccc}
	0 & -\lambda _{3}s_{3} & \lambda _{3}s_{2} \\ 
	\lambda _{2}s_{3} & 0 & -\lambda _{2}s_{1} \\ 
	-\lambda _{1}s_{2} & \lambda _{1}s_{1} & 0%
	\end{array}%
	\right] \leftrightarrow \hat{p}=\left( p_{1},p_{2},p_{3}\right) .
	\end{equation*}%
	Let two skew-symmetric matrices be 
	\begin{equation*}
	S=\left[ 
	\begin{array}{ccc}
	0 & -\lambda _{3}s_{3} & \lambda _{3}s_{2} \\ 
	\lambda _{2}s_{3} & 0 & -\lambda _{2}s_{1} \\ 
	-\lambda _{1}s_{2} & \lambda _{1}s_{1} & 0%
	\end{array}%
	\right] \text{ and }T=\left[ 
	\begin{array}{ccc}
	0 & -\lambda _{3}t_{3} & \lambda _{3}t_{2} \\ 
	\lambda _{2}t_{3} & 0 & -\lambda _{2}t_{1} \\ 
	-\lambda _{1}t_{2} & \lambda _{1}t_{1} & 0%
	\end{array}%
	\right] .
	\end{equation*}%
	($T\leftrightarrow \hat{q}=\left( q_{1},q_{2},q_{3}\right) $ ) Then 
	\begin{equation*}
	ST-TS=\left[ 
	\begin{array}{ccc}
	0 & \lambda _{1}\lambda _{3}\left( s_{2}t_{1}-s_{1}t_{2}\right)  & \lambda
	_{2}\lambda _{3}\left( s_{3}t_{1}-s_{1}t_{3}\right)  \\ 
	\lambda _{1}\lambda _{2}\left( s_{1}t_{2}-s_{2}t_{1}\right)  & 0 & \lambda
	_{2}\lambda _{3}\left( s_{3}t_{2}-s_{2}t_{3}\right)  \\ 
	\lambda _{1}\lambda _{2}\left( s_{1}t_{3}-s_{3}t_{1}\right)  & \lambda
	_{1}\lambda _{3}\left( s_{2}t_{3}-s_{3}t_{2}\right)  & 0%
	\end{array}%
	\right] \allowbreak 
	\end{equation*}%
	is found. Hence 
	\begin{eqnarray*}
		ST-TS &\leftrightarrow &\left( \lambda _{3}\left(
		s_{2}t_{3}-s_{3}t_{2}\right) ,\lambda _{2}\left(
		s_{3}t_{1}-s_{1}t_{3}\right) ,\lambda _{1}\left(
		s_{1}t_{2}-s_{2}t_{1}\right) \right)  \\
		&=&\hat{p}\wedge \hat{q}
	\end{eqnarray*}%
	is obtained. If 
	\begin{equation*}
	\cos \frac{\theta }{2}=a_{0},\text{ }s_{1}\sin \frac{\theta }{2}=a_{1},\text{
	}s_{2}\sin \frac{\theta }{2}=a_{2}\text{ and }s_{3}\sin \frac{\theta }{2}%
	=a_{3}
	\end{equation*}%
	then the matrix $\mathcal{A}dp$ is attained as follows:
	\begin{eqnarray*}
		\left[ \tiny{
			\begin{array}{cc}
				\cos ^{2}\frac{\theta }{2}+\left( \lambda _{1}\lambda _{2}s_{1}^{2}-\lambda
				_{1}\lambda _{3}s_{2}^{2}-\lambda _{2}\lambda _{3}s_{3}^{2}\right) \sin ^{2}%
				\frac{\theta }{2} & 2\lambda _{1}\lambda _{3}s_{1}s_{2}\sin ^{2}\frac{\theta 
				}{2}-2\lambda _{3}s_{3}\cos \frac{\theta }{2}\sin \frac{\theta }{2} \\ 
				2\lambda _{1}\lambda _{2}s_{1}s_{2}\sin ^{2}\frac{\theta }{2}+2\lambda
				_{2}s_{3}\cos \frac{\theta }{2}\sin \frac{\theta }{2} & \cos ^{2}\frac{%
					\theta }{2}+\left( -\lambda _{1}\lambda _{2}s_{1}^{2}+\lambda _{1}\lambda
				_{3}s_{2}^{2}-\lambda _{2}\lambda _{3}s_{3}^{2}\right) \sin ^{2}\frac{\theta 
				}{2} \\ 
				2\lambda _{1}\lambda _{2}s_{1}s_{3}\sin ^{2}\frac{\theta }{2}-2\lambda
				_{1}s_{2}\cos \frac{\theta }{2}\sin \frac{\theta }{2} & 2\lambda _{1}\lambda
				_{3}s_{2}s_{3}\sin ^{2}\frac{\theta }{2}+2\lambda _{1}s_{1}\cos \frac{\theta 
				}{2}\sin \frac{\theta }{2}%
			\end{array}%
		} \right.  \\ \\
		\left. \tiny{
			\begin{array}{c}
				2\lambda _{2}\lambda _{3}s_{1}s_{3}\sin ^{2}\frac{\theta }{2}+2\lambda
				_{3}s_{2}\cos \frac{\theta }{2}\sin \frac{\theta }{2} \\ 
				2\lambda _{2}\lambda _{3}s_{2}s_{3}\sin ^{2}\frac{\theta }{2}-2\lambda
				_{2}s_{1}\cos \frac{\theta }{2}\sin \frac{\theta }{2} \\ 
				\cos ^{2}\frac{\theta }{2}+\left( -\lambda _{1}\lambda _{2}s_{1}^{2}-\lambda
				_{1}\lambda _{3}s_{2}^{2}+\lambda _{2}\lambda _{3}s_{3}^{2}\right) \sin ^{2}%
				\frac{\theta }{2}%
		\end{array} }
		\right] .
	\end{eqnarray*}%
	Let us edit above expression, $\mathcal{A}dp $ is as follows
	\begin{eqnarray*}
		I+\left[ \tiny{
			\begin{array}{cc}
				\left( \lambda _{1}\lambda _{2}s_{1}^{2}-\lambda _{1}\lambda
				_{3}s_{2}^{2}-\lambda _{2}\lambda _{3}s_{3}^{2}-1\right) \sin ^{2}\frac{%
					\theta }{2} & 2\lambda _{1}\lambda _{3}s_{1}s_{2}\sin ^{2}\frac{\theta }{2}%
				-\lambda _{3}s_{3}\sin \theta  \\ 
				2\lambda _{1}\lambda _{2}s_{1}s_{2}\sin ^{2}\frac{\theta }{2}+\lambda
				_{2}s_{3}\sin \theta  & \left( -\lambda _{1}\lambda _{2}s_{1}^{2}+\lambda
				_{1}\lambda _{3}s_{2}^{2}-\lambda _{2}\lambda _{3}s_{3}^{2}-1\right) \sin
				^{2}\frac{\theta }{2} \\ 
				2\lambda _{1}\lambda _{2}s_{1}s_{3}\sin ^{2}\frac{\theta }{2}-\lambda
				_{1}s_{2}\sin \theta  & 2\lambda _{1}\lambda _{3}s_{2}s_{3}\sin ^{2}\frac{%
					\theta }{2}+\lambda _{1}s_{1}\sin \theta 
			\end{array}%
		}
		\right.  \\ \\
		\left. \tiny{
			\begin{array}{c}
				2\lambda _{2}\lambda _{3}s_{1}s_{3}\sin ^{2}\frac{\theta }{2}+\lambda
				_{3}s_{2}\sin \theta  \\ 
				2\lambda _{2}\lambda _{3}s_{2}s_{3}\sin ^{2}\frac{\theta }{2}-\lambda
				_{2}s_{1}\sin \theta  \\ 
				\left( -\lambda _{1}\lambda _{2}s_{1}^{2}-\lambda _{1}\lambda
				_{3}s_{2}^{2}+\lambda _{2}\lambda _{3}s_{3}^{2}-1\right) \sin ^{2}\frac{%
					\theta }{2}%
			\end{array}%
		} \right] 
	\end{eqnarray*}%
	Here if we use the equation
	\begin{equation*}
	2\sin ^{2}\frac{\theta }{2}=1-\cos \theta \text{ ve }\lambda _{1}\lambda
	_{2}s_{1}^{2}+\lambda _{1}\lambda _{3}s_{2}^{2}+\lambda _{2}\lambda
	_{3}s_{3}^{2}=1
	\end{equation*}%
	then 
	\begin{eqnarray*}
		\mathcal{A}dp &=&I+\sin \theta \left[ {\small {%
				\begin{array}{ccc}
					0 & -\lambda _{3}s_{3} & \lambda _{3}s_{2} \\ 
					\lambda _{2}s_{3} & 0 & -\lambda _{2}s_{1} \\ 
					-\lambda _{1}s_{2} & \lambda _{1}s_{1} & 0%
				\end{array}%
		}}\right] +\left( 1-\cos \theta \right)  \\
		&.&\left[ {\small {%
				\begin{array}{ccc}
					-\lambda _{1}\lambda _{3}s_{2}^{2}-\lambda _{2}\lambda _{3}s_{3}^{2} & 
					\lambda _{1}\lambda _{3}s_{1}s_{2} & \lambda _{2}\lambda _{3}s_{1}s_{3} \\ 
					\lambda _{1}\lambda _{2}s_{1}s_{2} & -\lambda _{1}\lambda
					_{2}s_{1}^{2}-\lambda _{2}\lambda _{3}s_{3}^{2} & \lambda _{2}\lambda
					_{3}s_{2}s_{3} \\ 
					\lambda _{1}\lambda _{2}s_{1}s_{3} & \lambda _{1}\lambda _{3}s_{2}s_{3} & 
					-\lambda _{1}\lambda _{2}s_{1}^{2}-\lambda _{1}\lambda _{3}s_{2}^{2}%
				\end{array}%
		}}\right] 
	\end{eqnarray*}%
	is obtained which means that it is 
	\begin{equation*}
	\mathcal{A}dp=I+\sin \theta S+\left( 1-\cos \theta \right) S^{2}.
	\end{equation*}%
	$i\in \left\{ 1,2,3\right\} ,$ the case of $\lambda _{i}=1$, $\mathcal{A}dp$
	is the matrix that have it made rotation through the angle $\theta $ around an
	axis on $\mathbb{R}^{3}$ 
\end{proof}

\subsubsection{Lie multiplication}

For point $e$ of the Lie group \newline
$S_{\mathbb{K}}=\left\{ p\in \mathbb{K}:N_{p}=1\right\} $, let us show the
set of left invariant vector fields as
\begin{equation*}
\mathcal{X}_{l}\left( S_{\mathbb{K}}\right) =\left\{ X\in \mathcal{X}\left(
S_{\mathbb{K}}\right)| \left( l_{p}\right) _{\ast }\left( X\right) =X\right\}
\end{equation*}%
$\mathcal{X}_{l}\left( S_{\mathbb{K}}\right) $ is isomorphic to tangent
space at point $e$. In that case \linebreak $\mathcal{X}_{l}\left( S_{%
	\mathbb{K}}\right) \cong T_{S_{\mathbb{K}}}\left( e\right) $. The following
multiplication is a Lie multiplication : 
\begin{equation*}
\begin{array}{cc}
\left[ ,\right] : & T_{S_{\mathbb{K}}}\left( e\right) \times T_{S_{\mathbb{K}%
}}\left( e\right) \rightarrow T_{S_{\mathbb{K}}}\left( e\right) \hspace*{4cm}
\\ 
& \left( X,Y\right) \rightarrow \left[ X,Y\right] =D_{X}Y-D_{Y}X%
\end{array}%
\end{equation*}%
where $D_{X}Y$ is covariant derivative of $Y$ according to $X$. $\left(
T_{S_{\mathbb{K}}}\left( e\right) ,\left[ ,\right] \right) $ is Lie algebra
of Lie group $S_{\mathbb{K}}$. Let us find the rule of the Lie
multiplication: \linebreak At point $s=0$, let us take a curve passing
through point $e$ that is $\gamma _{1}^{^{\prime }}\left( 0\right) =e_{1}~$: 
\begin{equation*}
\begin{array}{cc}
\gamma _{1}: & I\rightarrow G \hspace*{0.6cm} \\ 
& s\rightarrow \gamma _{1}\left( s\right)%
\end{array}%
\end{equation*}%
Let $p\in S_{\mathbb{K}}$. We have 
\begin{equation*}
\begin{array}{cc}
\vartheta _{1}: & I\rightarrow G \hspace*{0.6cm} \\ 
& s\rightarrow \vartheta _{l}\left( s\right)%
\end{array}%
\end{equation*}%
so that $\left( l_{p}\right) \left( \gamma _{1}\left( s\right) \right)
=\vartheta _{1}\left( s\right) $ and $\left( l_{p}\right) _{\ast }\left(
\gamma _{1}^{^{\prime }}\left( 0\right) \right) =\vartheta _{1}^{^{\prime
}}\left( 0\right) $. If taken \newline
$p=a_{0}+a_{1}e_{1}+a_{2}e_{2}+a_{3}e_{3}$ into the equation 
\begin{equation*}
\left( l_{p}\right) _{\ast }\left( \gamma _{1}^{^{\prime }}\left( 0\right)
\right) =\vartheta _{1}^{^{\prime }}\left( 0\right)
\end{equation*}%
then we get 
\begin{equation*}
\vartheta _{1}^{^{\prime }}\left( 0\right) =pe_{1}=-\lambda _{1}\lambda
_{2}a_{1}+a_{0}e_{1}+\lambda _{2}a_{3}e_{2}-\lambda _{1}a_{2}e_{3}=X_{1}.
\end{equation*}%
Similarly, we have 
\begin{equation*}
\vartheta _{2}^{^{\prime }}\left( 0\right) =pe_{2}=-\lambda _{1}\lambda
_{3}a_{2}-\lambda _{3}a_{3}e_{1}+a_{0}e_{2}+\lambda _{1}a_{1}e_{3}=X_{2}
\end{equation*}%
and 
\begin{equation*}
\vartheta _{3}^{^{\prime }}\left( 0\right) =pe_{3}=-\lambda _{2}\lambda
_{3}a_{2}+\lambda _{3}a_{2}e_{1}-\lambda _{2}a_{1}e_{2}+a_{0}e_{3}=X_{3}.
\end{equation*}%
where $\vartheta _{1},\vartheta _{2} : I\rightarrow G$. Thus a base of $\mathcal{X}_{l}\left( S_{\mathbb{K}}\right) $ is $\left\{
X_{1},X_{2},X_{3}\right\} $. The base vectors is written as a linear
cambination of base vectors of $\mathcal{X}_{l}\left( E^{4}\right) $: 
\begin{eqnarray*}
	X_{1} &=&-\lambda _{1}\lambda _{2}a_{1}\frac{\partial }{\partial x_{0}}+a_{0}%
	\frac{\partial }{\partial x_{1}}+\lambda _{2}a_{3}\frac{\partial }{\partial
		x_{2}}-\lambda _{1}a_{2}\frac{\partial }{\partial x_{3}}, \\
	X_{2} &=&-\lambda _{1}\lambda _{3}a_{2}\frac{\partial }{\partial x_{0}}%
	-\lambda _{3}a_{3}\frac{\partial }{\partial x_{1}}+a_{0}\frac{\partial }{%
		\partial x_{2}}+\lambda _{1}a_{1}\frac{\partial }{\partial x_{3}}, \\
	X_{3} &=&-\lambda _{2}\lambda _{3}a_{3}\frac{\partial }{\partial x_{0}}%
	+\lambda _{3}a_{2}\frac{\partial }{\partial x_{1}}-\lambda _{2}a_{1}\frac{%
		\partial }{\partial x_{2}}+a_{0}\frac{\partial }{\partial x_{3}}.
\end{eqnarray*}%
Now Let us define the Bracket operator on $\mathcal{X}_{l}\left( S_{\mathbb{K%
}}\right) $. For this purpose, giving multiplication rule of the bases of
set $\mathcal{X}_{l}\left( S_{\mathbb{K}}\right) $ will be enough: 
\begin{eqnarray*}
	D_{X_{1}}X_{2} &=&\left( -\lambda _{1}\lambda _{2}\lambda _{3}a_{3},\lambda
	_{1}\lambda _{3}a_{2},-\lambda _{1}\lambda _{2}a_{1},\lambda _{1}a_{0}\right)
	\\
	D_{X_{2}}X_{1} &=&\left( \lambda _{1}\lambda _{2}\lambda _{3}a_{3},-\lambda
	_{1}\lambda _{3}a_{2},\lambda _{1}\lambda _{2}a_{1},-\lambda _{1}a_{0}\right)
\end{eqnarray*}%
is found. If we use 
\begin{equation*}
\left[ X_{1},X_{2}\right] =D_{X_{1}}X_{2}-D_{X_{2}}X_{1}
\end{equation*}%
then we obtain 
\begin{equation*}
\left[ X_{1},X_{2}\right] =\left( -2\lambda _{1}\lambda _{2}\lambda
_{3}a_{3},2\lambda _{1}\lambda _{3}a_{2},-2\lambda _{1}\lambda
_{2}a_{1},2\lambda _{1}a_{0}\right) =2\lambda _{1}X_{3}.
\end{equation*}%
In the same way 
\begin{eqnarray*}
	D_{X_{2}}X_{3} &=&\left( -\lambda _{1}\lambda _{2}\lambda _{3}a_{1},\lambda
	_{3}a_{0},\lambda _{2}\lambda _{3}a_{3},-\lambda _{1}\lambda
	_{3}p_{2}\right) , \\
	D_{X_{3}}X_{2} &=&\left( \lambda _{1}\lambda _{2}\lambda _{3}a_{1},-\lambda
	_{3}a_{0},-\lambda _{2}\lambda _{3}a_{3},\lambda _{1}\lambda _{3}a_{2}\right)
\end{eqnarray*}%
are obtained. Also from these equations 
\begin{equation*}
\left[ X_{2},X_{3}\right] =2\lambda _{3}X_{1}
\end{equation*}%
is found. Finally, from the equations 
\begin{eqnarray*}
	D_{X_{3}}X_{1} &=&\left( -\lambda _{1}\lambda _{2}\lambda _{3}a_{2},-\lambda
	_{2}\lambda _{3}a_{0},\lambda _{2}a_{0},\lambda _{1}\lambda _{2}a_{1}\right)
	\\
	D_{X_{1}}X_{3} &=&\left( \lambda _{1}\lambda _{2}\lambda _{3}a_{2},\lambda
	_{2}\lambda _{3}a_{0},-\lambda _{2}a_{0},-\lambda _{1}\lambda
	_{2}a_{1}\right)
\end{eqnarray*}%
we achieved 
\begin{equation*}
\left[ X_{3},X_{1}\right] =2\lambda _{2}X_{2}.
\end{equation*}%
Since $\mathcal{X}_{l}\left( S_{\mathbb{K}}\right) \cong T_{S_{\mathbb{K}%
}}\left( e\right) $, we can give the Bracket multiplication rule on $T_{S_{%
		\mathbb{K}}}\left( e\right) $. Here provided that 
\begin{equation*}
e_{i}=\left. \left( \frac{\partial }{\partial x_{i}}\right) \right\vert _{e}%
\text{ }\left( i=1,2,3\right)
\end{equation*}%
respectively 
\begin{equation*}
\left[ X_{1},X_{2}\right] \mid _{e}=2\lambda _{1}\left( X_{3}\right) \mid
_{e}
\end{equation*}%
and 
\begin{equation*}
\left[ \left( X_{1}\right) \mid _{e},\left( X_{2}\right) \mid _{e}\right]
=2\lambda _{1}e_{3}
\end{equation*}%
is found. Therefore 
\begin{equation*}
\left[ e_{1},e_{2}\right] =2\lambda _{1}e_{3}
\end{equation*}%
is obtained. In a similar way the following equations are obtained: 
\begin{equation*}
\left[ e_{2},e_{3}\right] =2\lambda _{3}e_{1}\ \text{and}\ \left[ e_{3},e_{1}%
\right] =2\lambda _{2}e_{2}.
\end{equation*}%
\subsubsection{Matrix representation for Lie Algebra of $S_{\mathbb{K}}$}

Let us $X\in T_{S_{\mathbb{K}}}\left( e\right) $ and define mapping 
\begin{equation*}
\begin{array}[t]{ll}
Ad_{X}:T_{S_{\mathbb{K}}}\left( e\right) & \rightarrow T_{S_{\mathbb{K}%
}}\left( e\right) \\ 
\hspace{1.75cm}Y & \rightarrow Ad_{X}\left( Y\right) =\left[ X,Y\right] .%
\end{array}%
\end{equation*}%
According to the mapping, the matrix that corresponding to the linear mapping $Ad_{X}$
is the matrix notation of Lie algebra $S_{\mathbb{K}}$.

\begin{thm}
	Let $X=x_{1}e_{1}+x_{2}e_{2}+x_{3}e_{3}~$. 
	\begin{equation*}
	Ad_{X}=\left[ 
	\begin{array}{ccc}
	0 & -2\lambda _{3}x_{3} & 2\lambda _{3}x_{2} \\ 
	2\lambda _{2}x_{3} & 0 & -2\lambda _{2}x_{1} \\ 
	-2\lambda _{1}x_{2} & 2\lambda _{1}x_{1} & 0%
	\end{array}
	\right] .
	\end{equation*}
\end{thm}

\begin{proof}
	Let $X=x_{1}e_{1}+x_{2}e_{2}+x_{3}e_{3}$. Let us find the matrix corresponding to the linear mapping. If we write as
	\begin{equation*}
	Ad_{X}\left( e_{1}\right) =\left[ X,e_{1}\right] =\left[
	x_{1}e_{1}+x_{2}e_{2}+x_{3}e_{3},e_{1}\right]
	\end{equation*}%
	then since
	\begin{equation*}
	\left[ e_{1},e_{2}\right] =2\lambda _{1}e_{3},\text{ }\left[ e_{2},e_{3}%
	\right] =2\lambda _{3}e_{1}\text{ ve }\left[ e_{3},e_{1}\right] =2\lambda
	_{2}e_{2}
	\end{equation*}%
	and 
	\begin{equation*}
	\left[ e_{1},e_{1}\right] =\left[ e_{2},e_{2}\right] =\left[ e_{3},e_{3}%
	\right] =0
	\end{equation*}%
	also the Lie multiplication is linear, we obtain
	\begin{equation*}
	\left[ X,e_{1}\right] =2\lambda _{2}x_{3}e_{2}-2\lambda _{1}x_{2}e_{3}.
	\end{equation*}%
	Similarly
	\begin{eqnarray*}
		Ad_{X}\left( e_{2}\right) &=&\left[ X,e_{2}\right] \\
		&=&\left[ x_{1}e_{1}+x_{2}e_{2}+x_{3}e_{3},e_{2}\right] \\
		&=&-2\lambda _{3}x_{3}e_{1}+2\lambda _{1}x_{1}e_{3}, \\
		Ad_{X}\left( e_{3}\right) &=&\left[ X,e_{3}\right] \\
		&=&\left[ x_{1}e_{1}+x_{2}e_{2}+x_{3}e_{3},e_{3}\right] \\
		&=&-2\lambda _{3}x_{2}e_{1}-2\lambda _{2}x_{1}e_{2}.
	\end{eqnarray*}%
	Therefore the matrix that corresponding to linear map $Ad_{X}$ is the matrix of Lie algebra $S_{\mathbb{K}}$. And this matrix is
	\begin{equation*}
	Ad_{X}=\left[ 
	\begin{array}{ccc}
	0 & -2\lambda _{3}x_{3} & 2\lambda _{3}x_{2} \\ 
	2\lambda _{2}x_{3} & 0 & -2\lambda _{2}x_{1} \\ 
	-2\lambda _{1}x_{2} & 2\lambda _{1}x_{1} & 0%
	\end{array}%
	\right].
	\end{equation*}%
\end{proof} 
\subsubsection{Killing bi-linear form}

It is well known that the Killing form is a specific bi-linear form on a
finite-dimensional Lie algebra, defined by W. Killing. The following mapping
is called Killing bi-linear form of the Lie group $S_{\mathbb{K}}$: 
\begin{equation*}
\begin{array}{cc}
\mathcal{K}: & T_{S_{\mathbb{K}}}\left( e\right) \times T_{S_{\mathbb{K}%
}}\left( e\right) \rightarrow T_{S_{\mathbb{K}}}\left( e\right) \hspace*{4.5cm}
\\ 
& \left( X,Y\right) \rightarrow \mathcal{K}\left( X,Y\right) =tr\left(
Ad_{X}Ad_{Y}\right). 
\end{array}%
\end{equation*}%
The mapping $\mathcal{K}$ has the following properties:\newline
i. $\mathcal{K}$ is bi-linear,\newline
ii. $\mathcal{K}\left( X,Y\right) =\mathcal{K}\left( Y,X\right)$,\newline
iii. $\mathcal{K}\left( X,Y\right) =\mathcal{K}\left( Ad_{X},~Ad_{Y}\right) $.

\begin{thm}
	If 
	\begin{equation*}
	\begin{array}{cc}
	f: & \text{Im}\left( \mathbb{K}\right) \times \text{Im}\left( \mathbb{K}%
	\right) \rightarrow \mathbb{R}\hspace*{8.1cm} \\ 
	& \left( X,Y\right) \rightarrow f\left( X,Y\right) =\lambda _{1}\lambda
	_{2}x_{1}y_{1}+\lambda _{1}\lambda _{3}x_{2}y_{2}+\lambda _{2}\lambda
	_{3}x_{3}y_{3}%
	\end{array}%
	,
	\end{equation*}%
	then 
	\begin{equation*}
	\mathcal{K}\left( X,Y\right) =-8f\left( X,Y\right) .
	\end{equation*}
\end{thm}

\begin{proof}
	Let
	$X=x_{1}e_{1}+x_{2}e_{2}+x_{3}e_{3}~$ and $Y=y_{1}e_{1}+y_{2}e_{2}+y_{3}e_{3}$. Since
	\begin{equation*}
	Ad_{X}=\left[ 
	\begin{array}{ccc}
	0 & -2\lambda _{3}x_{3} & 2\lambda _{3}x_{2} \\ 
	2\lambda _{2}x_{3} & 0 & -2\lambda _{2}x_{1} \\ 
	-2\lambda _{1}x_{2} & 2\lambda _{1}x_{1} & 0%
	\end{array}%
	\right]
	\end{equation*}%
	$\allowbreak $and 
	\begin{equation*}
	Ad_{Y}=\left[ 
	\begin{array}{ccc}
	0 & -2\lambda _{3}y_{3} & 2\lambda _{3}y_{2} \\ 
	2\lambda _{2}y_{3} & 0 & -2\lambda _{2}y_{1} \\ 
	-2\lambda _{1}y_{2} & 2\lambda _{1}y_{1} & 0%
	\end{array}%
	\right]
	\end{equation*}%
	we obtain $Ad_{X}Ad_{Y}$ as
	\begin{equation*}
	\left[ 
	\begin{array}{ccc}
	\text{-}4\lambda _{1}\lambda _{3}x_{2}y_{2}\text{-}4\lambda _{2}\lambda
	_{3}x_{3}y_{3} & 4\lambda _{1}\lambda _{3}x_{2}y_{1} & 4\lambda _{2}\lambda
	_{3}x_{3}y_{1} \\ 
	4\lambda _{1}\lambda _{2}x_{1}y_{2} & \text{-}4\lambda _{1}\lambda
	_{2}x_{1}y_{1}\text{-}4\lambda _{2}\lambda _{3}x_{3}y_{3} & 4\lambda
	_{2}\lambda _{3}x_{3}y_{2} \\ 
	4\lambda _{1}\lambda _{2}x_{1}y_{3} & 4\lambda _{1}\lambda _{3}x_{2}y_{3} & 
	\text{-}4\lambda _{1}\lambda _{2}x_{1}y_{1}\text{-}4\lambda _{1}\lambda
	_{3}x_{2}y_{2}%
	\end{array}%
	\right]. \allowbreak
	\end{equation*}%
	Sum of diagonal elements of the matrix $Ad_{X}Ad_{Y}$ is
	\begin{equation*}
	tr\left( Ad_{X}Ad_{Y}\right) =-8\left( \lambda _{1}\lambda
	_{2}x_{1}y_{1}+\lambda _{1}\lambda _{3}x_{2}y_{2}+\lambda _{2}\lambda
	_{3}x_{3}y_{3}\right).
	\end{equation*}%
	Thus
	\begin{equation*}
	tr\left( Ad_{X}Ad_{Y}\right) =-8f\left( X,Y\right)
	\end{equation*}%
	is obtained.
\end{proof}

\begin{thm}
	Let $i\in \left\{ 1,2,3\right\}$ $\lambda _{i}>0$. \ $S_{\mathbb{K}%
	}=\left\{ p\in \mathbb{K}:N_{p}=1\right\} $ is compact.
\end{thm}

\begin{proof}
	If $\mathcal{K}\left( X,X\right) <0$, then the Lie group is compact. Since $i\in
	\left\{ 1,2,3\right\} ,$ $\lambda _{i}>0$, we obtain $f\left(
	X,X\right) >0,$ relatively $\mathcal{K}\left( X,X\right) <0$. This gives us the desired.
\end{proof}

\begin{thm}
	Let $K$ be the matrix that corresponding to Killing bi-linear form of
	Lie group $S_{\mathbb{K}}$ and 
	\begin{equation*}
	\varepsilon =\left[ 
	\begin{array}{ccc}
	\lambda _{1}\lambda _{2} & 0 & 0 \\ 
	0 & \lambda _{1}\lambda _{3} & 0 \\ 
	0 & 0 & \lambda _{2}\lambda _{3}%
	\end{array}%
	\right] .
	\end{equation*}%
	Then $K=-8\varepsilon $.
\end{thm}

\begin{proof}
	
	The following mapping corresponds to killing bi-linear form of Lie group $S_{\mathbb{K}}$:
	\begin{equation*}
	\begin{array}{cc}
	\mathcal{K}: & T_{S_{\mathbb{K}}}\left( e\right) \times T_{S_{\mathbb{K}%
	}}\left( e\right) \rightarrow T_{S_{\mathbb{K}}}\left( e\right) \hspace*{4cm} \\ 
	& \left( X,Y\right) \rightarrow \mathcal{K}\left( X,Y\right) =-8f\left(
	X,Y\right)%
	\end{array}%
	\end{equation*}%
	and since $T_{S_{\mathbb{K}}}\left( e\right) \cong Sp\left\{
	e_{1},e_{2},e_{3}\right\} $
	\begin{equation*}
	K=\left[ 
	\begin{array}{ccc}
	\mathcal{K}\left( e_{1},e_{1}\right) & \mathcal{K}\left( e_{1},e_{2}\right)
	& \mathcal{K}\left( e_{1},e_{3}\right) \\ 
	\mathcal{K}\left( e_{2},e_{1}\right) & \mathcal{K}\left( e_{2},e_{2}\right)
	& \mathcal{K}\left( e_{2},e_{3}\right) \\ 
	\mathcal{K}\left( e_{3},e_{1}\right) & \mathcal{K}\left( e_{3},e_{2}\right)
	& \mathcal{K}\left( e_{3},e_{3}\right)%
	\end{array}%
	\right]
	\end{equation*}%
	is achieved. Therefore 
	\begin{eqnarray*}
		K &=&\left[ 
		\begin{array}{ccc}
			-8\lambda _{1}\lambda _{2} & 0 & 0 \\ 
			0 & -8\lambda _{1}\lambda _{3} & 0 \\ 
			0 & 0 & -8\lambda _{2}\lambda _{3}%
		\end{array}%
		\right] \\
		&=&-8\varepsilon
	\end{eqnarray*}%
	is obtained.
\end{proof}


\subsection*{Acknowledgment}

This article has been obtained from the Ph.D. dissertation of Tuncay Deniz \c{S}ent%
\"{u}rk under the supervision of Zafer \"{U}nal.



\begin{thebibliography}{99}
\bibitem{ham0} W. R. Hamilton, \textit{Researches respecting quaternions.}
First series, 1843. In Halberstam and Ingram \cite{hal}, Sec. 7, 159-226.
First published as \cite{ham3}.

\bibitem{ham3} W. R. Hamilton, \textit{Researches respecting quaternions.}
Transactions of the Royal Irish Academy \textbf{21} (1848), 199-296.

\bibitem{ham4} W. R. Hamilton, \textit{On a new species of Imaginary
quantities connected with the theory of quaternions}. Proceedings of the
Royal Irish Academy \textbf{2} (1844), 424-434.

\bibitem{ham1} W. R. Hamilton, \textit{Lectures on Quaternions.} Hodges and
Smith, Dublin, 1853.

\bibitem{ham5} W. R. Hamilton, \textit{Elements of Quaternions.} London,
U.K.:Longmans Green, 1866.

\bibitem{ham2} W. R. Hamilton, \textit{On the geometrical interpretation of
some results obtained by calculation with biquaternions.} In: Halberstam and
Ingram \cite{hal}, Sec. 35, 424-425. First Published in preceedings of the
Royal Irish Academy, 1853.

\bibitem{hal} R.E. Halberstam, Ingram (Eds.), \textit{The mathmetical Papers
of Sir William Rowan Hamilton.} 3 Algebra, Cambridge University Press,
Cambridge, 1967.

\bibitem{war} J. P. Ward, \textit{Quaternions and Cayley Numbers Algebra and
Applications.} Kluwer Academic Publishers, London 1997, 54-102.

\bibitem{coc} J. Cockle, \textit{On Systems of algebra involving more than
one imaginary; and on equations of the fifth degree.} Philosophical Magazine 
\textbf{35:238} (1849), 434-437.

\bibitem{clif} W. Clifford, \textit{Preliminary sketch of biquaternions.}
Proc. of London Math. Soc. 10, 1873.

\bibitem{dic} L. E. Dickson, \textit{On the Theory of Numbers and Generalized Quaternions.} American Journal of Mathematics \textbf{46(1)} (1924), 1-16.

\bibitem{gri} L. W. Griffiths, \textit{Generalized Quaternion Algebras and the Theory of Numbers.} American Journal of Mathematics \textbf{50(2)} (1928), 303-314.

\bibitem{sen} T. D. \c{S}ent\"{u}rk, A. Da\c{s}demir, G. Bilgici and Z. \"{U}%
nal, \textit{On Unrestricted Horadam Generalized Quaternions.} Utilitas
Mathematica \textbf{110} (2019), 89-98.

\bibitem{ahm} A. Da\c{s}demir, G. Bilgici, \textit{Gaussian Mersenne numbers
and generalized Mersenne quaternions.} Notes on Number Theory and Discrete
Mathematics \textbf{25(3)} (2019), 87-96. Doi: 10.7546/nntdm.2019.25.3.87-96

\bibitem{tok} \"{U}. Toke\c{s}er, Z. \"{U}nal, G. Bilgici, \textit{Split
Pell and Pell-Lucas Quaternions.} Advances in Applied Clifford Algebras 
\textbf{27(2)} (2017), 1881-1893.

\bibitem{bil} G. Bilgici, \"{U}. Toke\c{s}er, Z. \"{U}nal. \textit{\
k-Fibonacci and k-Lucas Generalized Quaternions.} Konuralp Journal of
Mathematics \textbf{5(2)} (2017) 102-113.

\bibitem{tas} D. Ta\c{s}c\i, F. Yal\c{c}\i n, \textit{Fibonacci-p
Quaternions.} Advances in Applied Clifford Algebras \textbf{25(1)} (2015),
245-254.

\bibitem{swa} M. N. S. Swamy, \textit{On generalized Fibonacci quaternions.}
The Fibonacci Quarterly \textbf{11(5)} (1973), 547-550.

\bibitem{hor} A. F. Horadam, \textit{Complex Fibonacci numbers and Fibonacci
quaternions.} The American Mathematical Monthly \textbf{70(3)} (1963),
289-291.

\bibitem{kar} A. Karger, J. Novak, \textit{Space kinematics and Lie groups},
Gordon and science publishers, 1985.

\bibitem{jaf1} M. Jafari, Y. Yayl\i, \textit{Generalized Quaternions and
Their Algebraic Properties.} Commun. Fac. Sci. Univ. Ank. Series A1 \textbf{%
64(1)} (2015), 15-27. Doi: 10.1501/Commual\_0000000724.

\bibitem{cho} E. Cho, \textit{De Moivre's Formula for Quaternions.} Applied
Mathematics Letters \textbf{11(6)} (1998), 33-35.

\bibitem{kab} H. Kabaday\i, Y. Yayl\i, \textit{De Moivre's Formula for Dual
Quaternions.} Kuwait journal of science and technology \textbf{38(1)}
(2011), 15-23.

\bibitem{ozd} M. \"{O}zdemir, \textit{The roots of a Split Quaternion.}
Applied mathmematics letters \textbf{22} (2009), 258-263.

\bibitem{jaf2} M. Jafari, M. Meral, Y. Yayl\i . \textit{Matrix
representation of dual quaternions.} Gazi University journal of science 
\textbf{26(4)} (2013), 535-542.

\bibitem{jaf3} M. Jafari, H. Mortazaasl, Y. Yayl\i, \textit{De Moivre's
Formula for Matrices of Quaternions.} JP Journal of Algebra, Number Theory
and Applications \textbf{21(1)} (2011), 57-67.

\bibitem{mer} M. Meral, \textit{Kuaterniyonlara ait matrisler i\c{c}in
De'Moivre ve Euler Form\"{u}lleri.} Y\"{u}ksek Lisans Tezi, Ankara \"{U}%
niversitesi 2009.

\bibitem{jaf4} A. B. Mamagani, M. Jafari, \textit{On Properties of
Generalized Quaternions Algebra.} Journal of Novel Applied Sciences \textbf{%
2(12)} (2013), 683-689.

\bibitem{agr} O. P. Agrawal, \textit{Hamilton operators and dual-number
quaternions in spatial kinematics.} Mechanism and machine theory \textbf{22(6%
}) (1987), 569-575.

\bibitem{jaf5} M. Jafari, Y. Yayl\i, \textit{Hamilton operators and
generalized quaternions.} 8.Geometri Sempozyumu, Akdeniz \"{U}niversitesi,
Antalya 2010.

\bibitem{jaf6} M. Jafari, \textit{Matrices of Generalized Dual Quaternions.}
Konuralp Journal of Mathematics \textbf{3(2)} (2015), 110-121.

\bibitem{olm} O. \"{O}lmez, \textit{Genelle\c{s}tirilmi\c{s} Kuaterniyonlar
ve Uygulamalar\i.} Y\"{u}ksek Lisans Tezi, Ankara \"{U}niversitesi 2006.
\end{thebibliography}
\end{document}